\documentclass[10pt,twoside]{amsart}
\usepackage{amsmath, amsthm, amscd, amsfonts, amssymb, graphicx}
\newtheorem{thm}{Theorem}[section]

\newtheorem{prop}[thm]{Proposition}
\newtheorem{defn}[thm]{Definition}

\newtheorem{problem}[thm]{Example}
\numberwithin{equation}{section}

\begin{document}


\begin{center}\large{{\bf{Some operators on interval-valued Hesitant fuzzy soft sets}}}

\vspace{0.5cm}

Manash Jyoti Borah$^{1}$ and Bipan Hazarika$^{2\ast}$

\vspace{.2cm}
$^{1}$Department of Mathematics, Bahona College,  Jorhat-785 101, Assam, India\\
Email:mjyotibora9@gmail.com

\vspace{.2cm}

$^{2}$Department of Mathematics, Rajiv Gandhi University, Rono Hills, Doimukh-791 112, Arunachal Pradesh, India\\

Email:  bh\_rgu$@$yahoo.co.in
\end{center}
\vspace{.5cm}
\title{}
\author{}
\thanks{{$^{\ast}$The corresponding author}}

\begin{abstract}
 The main aim of this paper is to introduced the operations "Union" and "Intersection," and four operators $O_1, O_2, O_3, O_4 $  on interval-valued hesitant fuzzy soft sets and discuss some of their properties. \\


AMS subject classification no: 03E72.
\end{abstract}

\maketitle
\pagestyle{myheadings}
\markboth{\rightline {\scriptsize Borah, Hazarika}}
         {\leftline{\scriptsize Some operators on interval-valued...}}

\maketitle

\section{Introduction}
Interval arithmetic was first suggested by Dwyer \cite{Dwyer51} in
1951. Development of interval arithmetic as a formal system and evidence of
its value as a computational device was provided by Moore \cite{Moore} in 1959 and Moore and Yang \cite{MooreYang62} in 1962. Further  works on interval numbers can be found in  Dwyer \cite{Dwyer53}, Fischer \cite{Fischer}. Furthermore, Moore
and Yang \cite{MooreYang58}, have developed applications to differential equations. Chiao in \cite{Chiao} introduced sequence of interval numbers and
defined usual convergence of sequences of interval number.\newline

A set consisting of a closed interval of real numbers x such that $a\leq x\leq
b$ is called an interval number. A real interval can also be considered as a
set. Thus we can investigate some properties of interval numbers, for instance
arithmetic properties or analysis properties.We denote the set of all real
valued closed intervals by I$%
\mathbf{R}
.$ Any elements of I$%
\mathbf{R}
$ is called closed interval and denoted by $\overline{x}.$ That is
$\overline{x}=\left\{  x\in%
\mathbb{R}
:\text{ }a\leq x\leq b\right\}  .$ An interval number $\overline{x}$ is a
closed subset of real numbers \cite{Chiao}. Let $x_{l}$ and $x_{r}$ be
first and last points of $\overline{x}$ interval number, respectively. For
$\overline{x}_{1},\overline{x}_{2}\in$I$%
\mathbf{R}
,$ we have $\overline{x}_{1}=\overline{x}_{2}\Leftrightarrow x_{1_{l}}%
$=$x_{2_{l}}$,$x_{1_{r}}$=$x_{2_{r}}.$ $\overline{x}_{1}+\overline{x}%
_{2}=\left\{  x\in%
\mathbb{R}
:x_{1_{l}}+x_{2_{l}}\leq x\leq x_{1_{r}}+x_{2_{r}}\right\}  ,$and if
$\alpha\geq0,$ then $\alpha\overline{x}=\left\{  x\in%
\mathbb{R}
:\text{ }\alpha x_{1_{l}}\leq x\leq\alpha x_{1_{r}}\right\}  $ and if
$\alpha<0,$ then $\alpha\overline{x}=\left\{  x\in%
\mathbb{R}
:\text{ }\alpha x_{1_{r}}\leq x\leq\alpha x_{1_{l}}\right\}  ,$
\[
\overline{x}_{1}.\overline{x}_{2}=\left\{
\begin{array}
[c]{c}%
x\in%
\mathbb{R}
:\min\left\{  x_{1_{l}}.x_{2_{l}},x_{1_{l}}.x_{2_{r}},x_{1_{r}}.x_{2_{l}%
},x_{1_{r}}.x_{2_{r}}\right\}  \leq x\\
\leq\max\left\{  x_{1_{l}}.x_{2_{l}},x_{1_{l}}.x_{2_{r}},x_{1_{r}}.x_{2_{l}%
},x_{1_{r}}.x_{2_{r}}\right\}
\end{array}
\right\}  .
\]

The Hesitant fuzzy set, as one of the extensions of Zadeh \cite{zadeh} fuzzy set, allows the membership degree that an element to a set
presented by several possible values, and  it can express the hesitant information
more comprehensively than other extensions of fuzzy set. In 2009, Torra and Narukawa \cite{torra2} introduced the
concept of hesitant fuzzy set. In 2011, Xu and Xia \cite{xu}
defined the concept of hesitant fuzzy element, which can be considered as
the basic unit of a hesitant fuzzy set, and is a simple and effective tool used to express the
decision makers’ hesitant preferences in the process of decision making. So many  researchers  has done lots of research work on aggregation, distance,
similarity and correlation measures, clustering analysis, and decision making with
hesitant fuzzy information. In 2013, Babitha and John \cite{babitha} defined another important soft set Hesitant fuzzy soft sets. They introduced basic operations such as intersection, union, compliment  and De Morgan's law was proved.In 2013, Chen et al. \cite{chen} extended hesitant fuzzy sets into interval-valued hesitant fuzzy environment and introduced the concept of interval-valued hesitant fuzzy sets. In 2015, Zhang et al. \cite{zhang} introduced some operations such as complement, "AND","OR", ring sum and ring product on interval-valued hesitant fuzzy soft sets.    \\

There are many theories like theory of probability, theory of fuzzy sets, theory of
intuitionistic fuzzy sets, theory of rough sets etc. which can be considered as mathematical
tools for dealing with uncertain data, obtained in various fields of engineering, physics,
computer science, economics, social science, medical science, and of many other diverse
fields. But all these theories have their own difficulties. The most appropriate theory for
dealing with uncertainties is the theory of fuzzy sets, introduced by L.A. Zadeh \cite{zadeh} in 1965.
This theory brought a paradigmatic change in mathematics. But there exists difficulty, how to
set the membership function in each particular case. The theory of intuitionistic fuzzy sets (see \cite{atanassov1, atanassov}) is a
more generalized concept than the theory of fuzzy sets, but this theory has the same
difficulties. All the above mentioned theories are successful to some extent in dealing with
problems arising due to vagueness present in the real world. But there are also cases where
these theories failed to give satisfactory results, possibly due to inadequacy of the
parameterization tool in them. As a necessary supplement to the existing mathematical tools
for handling uncertainty,  Molodtsov \cite{molodstov} introduced the theory of soft sets as
a new mathematical tool to deal with uncertainties while
modelling the problems in engineering, physics, computer
science, economics, social sciences, and medical sciences.
Molodtsov et al \cite{MolodtsovLeonov} successfully applied soft sets in
directions such as smoothness of functions, game theory,
operations research, Riemann integration, Perron integration, probability, and theory of measurement. Maji et al \cite{majietal1} gave the first practical application of soft sets in
decision-making problems.  Maji et al \cite{majietal2} defined
and studied several basic notions of the soft set theory. Also \c{C}a\v{g}man et al \cite{cagman} studied several basic notions of the soft set theory.  V. Torra \cite{torra,torra2} and Verma and Sharma \cite{verma} discussed the relationship between hesitant fuzzy set and showed that the envelope of hesitant fuzzy set is an intuitionistic fuzzy set. Zhang et al \cite{zhang} introduced weighted interval-valued hesitant fuzzy soft sets and finally applied it in decision making problem. Thakur et al \cite{thakur} proposed four new operators $O_1, O_2, O_3, O_4 $ on hesitant fuzzy sets.\\

In this paper, in section 3, we study operations union and intersetion on hesitant interval-valued fuzzy soft sets and some interesting properties of this noton. In section 4, we introduce four operators  $O_1, O_2, O_3, O_4$ in interval-valued hesitant fuzzy soft sets. Also various proposition are proved by using them.   

\section{Preliminary Results}

In this section we recall some basic concepts and definitions
regarding fuzzy soft sets, hesitant fuzzy  set and hesitant fuzzy soft set.

\begin{defn}
\cite{maji} Let $U$ be an initial universe and $F$ be a set of parameters. Let $\tilde{P}(U)$ denote the power set of $U$ and $A$ be a
non-empty subset of $F.$  Then $F_A$ is called a
fuzzy soft set over U where $F:A\rightarrow \tilde{P}(U)$is a
mapping from $A$ into $\tilde{P}(U).$
\end{defn}

\begin{defn}
\cite{molodstov}  $F_E$ is called a soft
set over $U$ if and only if $F$ is a mapping of $E$ into the set of all
subsets of the set $U.$
\end{defn}
In other words, the soft set is a parameterized family of subsets
of the set $U.$ Every set $F(\epsilon),$ $\epsilon\tilde{\in} E,$
from this family may be considered as the set of
$\epsilon$-element of the soft set $F_E$  or as the set of
$\epsilon$-approximate elements of the soft set.
\begin{defn}
	\cite{atanassov,xu91} Let intuitionistic fuzzy value IFV(X) denote the family of all IFVs defined on the universe X, and let $ \alpha, \beta\in IFV(X) $ be given as: \\
	$ \alpha=(\mu_\alpha,\nu_\alpha), \beta=(\mu_\beta,\nu_\beta),$
	\begin{enumerate}
\item[	(i)] $ \alpha\cap \beta=(\min(\mu_\alpha,\mu_\beta ), \max(\nu_\alpha,\nu_\beta ) $
\item[	(ii)] $ \alpha\cup \beta=(\max(\mu_\alpha,\mu_\beta ), \min(\nu_\alpha,\nu_\beta ) $	
	\item[(iii)] $ \alpha \ast \beta=(\frac{\mu_\alpha+\mu_\beta}{2(\mu_\alpha.\mu_\beta+1)},\frac{\nu_\alpha+\nu_\beta}{2(\nu_\alpha.\nu_\beta+1)} ).$	
	\end{enumerate}

\end{defn}
\begin{defn}
	\cite{torra} Given a fixed set $X, $ then a hesitant fuzzy set (shortly HFS) in $X$ is in terms of a function that when applied to $X$ return a subset of $[0, 1].$ We express the HFS by a mathematical symbol:\\
	$ F=\{<h, \mu_{F}(x)> : h\in X\}, $
	where $ \mu_{F}(x) $ is a set of some values in $[0,1],$ denoting the possible membership degrees of the element $ h\in X $ to the set $ F.$ $ \mu_{F}(x) $ is called a hesitant fuzzy element (HFE) and $ H $ is the set of all HFEs. 
\end{defn}

\begin{defn}
	\cite{torra} Given an hesitant fuzzy set $F, $ define below it lower and upper bound as \\
	lower bound $ F^{-}(x)=\min F(x).$\\
	upper bound $ F^{+}(x)=\max F(x).$
\end{defn}

\begin{defn}
	\cite{torra} Let $ \mu_{1},  \mu_{2}\in H $ and three operations are defined as follows:
	\begin{enumerate}
\item[(1)] $ \mu_{1}^C=\cup_{\gamma_{1}\in\mu_{1}} \{1-\gamma_{1}\}; $
	
	\item[(2)] $ \mu_{1}\cup\mu_{2}=\cup_{\gamma_{1}\in\mu_{1}, \gamma_{2}\in\mu_{2}}\max \{\gamma_{1},\gamma_{2} \} ; $
	
	\item[(3)] $ \mu_{1}\cap\mu_{2}=\cap_{\gamma_{1}\in\mu_{1}, \gamma_{2}\in\mu_{2}}\min \{\gamma_{1},\gamma_{2} \}.$
			\end{enumerate}
\end{defn}

\begin{defn}
\cite{chen} Let $X$ be a reference set, and $D[0,1]$ be the set of all closed subintervals of $[0,1].$ An IVHFS on $X$ is 
	$ F=\{<h_i, \mu_{F}(h_i)> : h_i\in X, i=1,2,...n\}, $
	where $\mu_{F}(h_i):X\rightarrow D[0,1]$ denotes all possible interval-valued membership degrees of the element $ h_i\in X $ to the set F. For convenience, we call $ \mu_{F}(h_i) $ an interval -valued hesitant fuzzy element (IVHFE), which reads
	$\mu_{F}(h_i)=\{\gamma: \gamma\in\mu_{F}(h_i) \}.$\\
Here $ \gamma=[\gamma^L, \gamma^U] $ is an interval number. $ \gamma^L=\inf \gamma $ and $ \gamma^U=\sup \gamma $ represent the lower and upper limits of $ \gamma,$ respectively. An IVHFE is the basic unit of an IVHFS and it can be considered as a special case of the IVHFS. The relationship between IVHFE and IVHFS is similar to that between interval-valued fuzzy number and interval-valued fuzzy set. 
\end{defn}
\begin{problem}\label{exam36}
	Let $ U=\{h_1, h_2\} $ be a reference set and let $ \mu_F(h_1)=\{[0.6,0.8],[0.2,0.7]\},\mu_F(h_2)=\{[0.1,0.4]\}  $ be the IVHFEs of $ h_i (i=1,2) $ to a set F respectively. Then IVHFS  F can be written as 
	 $F=\{<h_1,\{[0.6, 0.8], [0.2, 0.7]\}>, <h_2, \{[0.1, 0.4]\}>\}.$
\end{problem}

\begin{defn}
	\cite{xu9} Let $ \tilde{a}=[\tilde{a}^L,\tilde{a}^U ] $ and 
$ \tilde{b}=[\tilde{b}^L,\tilde{b}^U ] $ be two interval numbers and $ \lambda\geq0, $ then 

\begin{enumerate}
	\item[(i)] $\tilde{a}=\tilde{b}\Leftrightarrow\tilde{a}^L=\tilde{b}^L$ and $ \tilde{a}^U=\tilde{b}^U ;$
	\item[(ii)]$\tilde{a}+\tilde{b}=[\tilde{a}^L+\tilde{b}^L, \tilde{a}^U+\tilde{b}^U] ;$
	\item[(iii)]$ \lambda\tilde{a}=[\lambda\tilde{a}^L,\lambda\tilde{a}^U ] ,$ especially $ \lambda\tilde{a}=0, $ if $ \lambda=0.$ 
\end{enumerate}
\end{defn}

\begin{defn}
	\cite{xu9} Let $ \tilde{a}=[\tilde{a}^L,\tilde{a}^U ] $ and 
	$ \tilde{b}=[\tilde{b}^L,\tilde{b}^U ], $ and let $ l_a=\tilde{a}^U-\tilde{a}^L $ and $ l_b=\tilde{b}^U-\tilde{b}^L; $ then the degree of possibility of $\tilde{a}\geq\tilde{b}  $ is formulated by\\
	$p(\tilde{a}\geq\tilde{b})=\textrm{max}\{1-\textrm{max}(\frac{\tilde{b}^U-\tilde{a}^L}{l_{\tilde{a}}+l_{\tilde{b}}},0),0\}.$\\
	Above equation is proposed in order to compare two interval numbers, and to rank all the input arguments.
\end{defn}
\begin{defn}
	\cite{chen} For an IVHFE $ \tilde{\mu} ,$ $ s(\tilde{\mu}) =\frac{1}{l_{ \tilde{\mu}}}\sum_{ \tilde{\gamma}\in \tilde{\mu}}\tilde{\gamma}$ is called the score function of $ \tilde{\mu} $ with $l_{ \tilde{\mu}}$ being the number of the interval values in  $ \tilde{\mu}, $ and $ s(\tilde{\mu}) $ is an interval value belonging to $[0,1].$ For two IVHFEs $\tilde{\mu_1}  $ and $\tilde{\mu_2} , $ if $  s(\tilde{\mu_1})\geq s(\tilde{\mu_2}),$ then $\tilde{\mu_1}\geq \tilde{\mu_2}. $\\
We can judge the magnitude of two IVHFEs using above equation. 
\end{defn}
\begin{defn}
\cite{chen} Let $\tilde{\mu}, \tilde{\mu_1}$ and $ \tilde{\mu_2}$ be three IVHFEs, then 
\begin{enumerate}
	\item[(i)]$\tilde{\mu}^C=\{[1-\tilde{\gamma}^U,1-\tilde{\gamma}^L]:\tilde{\gamma}\in\tilde{\mu} \} ; $
\item[(ii)] $\tilde{\mu_1}\cup\tilde{\mu_2}=\{[\textrm{max}(\tilde{\gamma_1}^L,\tilde{\gamma_2}^L),\textrm{max}(\tilde{\gamma_1}^U,\tilde{\gamma_2}^U)]:\tilde{\gamma_1}\in\tilde{\mu_1}, \tilde{\gamma_2}\in\tilde{\mu_2} \} ; $

\item[(iii] $\tilde{\mu_1}\cap\tilde{\mu_2}=\{[\textrm{min}(\tilde{\gamma_1}^L,\tilde{\gamma_2}^L),\textrm{min}(\tilde{\gamma_1}^U,\tilde{\gamma_2}^U)]:\tilde{\gamma_1}\in\tilde{\mu_1}, \tilde{\gamma_2}\in\tilde{\mu_2} \} ; $
\item[(iv)] $\tilde{\mu_1}\oplus \tilde{\mu_2} =\{[\tilde{\gamma_1}^L+\tilde{\gamma_2}^L-\tilde{\gamma_1}^L.\tilde{\gamma_2}^L, \tilde{\gamma_1}^U+\tilde{\gamma_2}^U-\tilde{\gamma_1}^U.\tilde{\gamma_2}^U]:\tilde{\gamma_1}\in\tilde{\mu_1}, \tilde{\gamma_2}\in\tilde{\mu_2} \} ;  $
\item[(v)] $\tilde{\mu_1}\otimes \tilde{\mu_2} =\{[\tilde{\gamma_1}^L.\tilde{\gamma_2}^L,\tilde{\gamma_1}^U.\tilde{\gamma_2}^U]:\tilde{\gamma_1}\in\tilde{\mu_1}, \tilde{\gamma_2}\in\tilde{\mu_2} \}.$  
\end{enumerate}
\end{defn}
\begin{prop}
\cite{chen} For three IVHFEs $ \tilde{\mu}, \tilde{\mu_1} $ and $ \tilde{\mu_2} ,$ we have 
\begin{enumerate}
	\item[(i)]$\tilde{\mu_1}^C\cup\tilde{\mu_2}^C= (\tilde{\mu_1}\cap\tilde{\mu_2})^C; $
	
	\item[(ii)]$ \tilde{\mu_1}^C\cap\tilde{\mu_2}^C= (\tilde{\mu_1}\cup\tilde{\mu_2})^C; $
\end{enumerate}
\end{prop}
\begin{defn}
	\cite{wang} Let $U$ be an initial universe and $E$ be a set of parameters. Let $\tilde{F}(U)$ be the set of all hesitant fuzzy subsets of $ U. $ Then $F_E$ is called a hesitant fuzzy soft set (HFSS) over $ U, $ where $\tilde {F}:E\rightarrow \tilde{F}(U). $\\
	A HFSS is a parameterized family of hesitant fuzzy subsets of $U,$ that is, $\tilde{F}(U).$ For all $\epsilon\tilde{\in} E,$  $F(\epsilon)$ is referred to as the set of $ \epsilon -$ approximate elements of the HFSS $F_E.$ It can be written as 	 $\tilde{F(\epsilon)}=\{<h, \mu_{\tilde{F(\epsilon)(x)}}> : h\in U\}.$\\
	 Since  HFE can represent the situation, in which different membership function are considered possible (see \cite{torra}), $ \mu_{\tilde{F(\epsilon)(x)}} $ is a set of several possible values, which is the hesitant fuzzy membership degree. In particular, if $ \tilde{F(\epsilon)} $ has only one element,  $ \tilde{F(\epsilon)} $  can be called a hesitant fuzzy soft number. For convenience, a hesitant fuzzy soft number (HFSN) is denoted by $ \{<h, \mu_{\tilde{F(\epsilon)(x)}}>\}.$ 
	 
\end{defn}
\begin{problem}\label{exam36}
		Suppose $ U=\{h_1,h_2\} $ be an initial universe and $ E=\{e_1, e_2, e_3, e_4\} $ be a set of parameters. Let $ A=\{e_1, e_2\} .$ Then the hesitant fuzzy soft set $ F_A $ is given as\\  $F_A=\{F(e_1)=\{<h_1,\{0.6, 0.8\}>, <h_2, \{0.8, 0.4, 0.9\}>\},  F(e_2)=\{<h_1,\{0.9, 0.1, 0.5\}>, <h_2,\{0.2 \}>\}.$
	
\end{problem}

\begin{defn}
\cite{zhang} Let $ (U,E) $ be a soft universe and $ A\subseteq E .$ Then $ F_A $ is called an interval valued hesitant fuzzy soft set over U, where $ F $ is a mapping given by $F:A \rightarrow IVHF(U). $\\
An interval-valued hesitant fuzzy soft set is a parameterized family of interval-valued hesitant fuzzy sub set of $ U. $ That is to say, $ F(e) $ is an interval-valued hesitant fuzzy subset in $ U,\forall e\in A. $ Following the standard notations, $ F(e) $ can be written as \\
  $\tilde{F(e)}=\{<h, \mu_{\tilde{F(e)(x)}}> : h\in U\}.$
\end{defn}

\begin{problem}\label{exam36}
		Suppose $ U=\{h_1,h_2\} $ be an initial universe and $ E=\{e_1, e_2, e_3, e_4\} $ be a set of parameters. Let $ A=\{e_1, e_2\} .$ Then the interval valued hesitant fuzzy  soft set $ F_A $ is given as\\
	$F_A=\{e_1=\{<h_1,[0.6, 0.8]>, <h_2, [0.1, 0.4]>\}\\  e_2=\{<h_1,[0.2, 0.6],[0.3,0.9] >, <h_2, [0.2,0.5],[0.2,0.8], [0.2,0.8]>\}.$
\end{problem}

\begin{defn}
\cite{zhang} $ U $ be an initial universe and let $ E $ be a set of parameters. Supposing that $ A, B \tilde{\subseteq}E, F_A $and $ F_B $ are two interval-valued hesitant fuzzy soft sets, one says that $ F_A $ is an interval-valued hesitant fuzzy soft subset of $ G_B $ if and only if 
\begin{enumerate}
\item[(i)] $ A\tilde{\subseteq}B, $
\item[(ii)]$\gamma_1^{\sigma(k)} \tilde{\leq}\gamma_2^{\sigma(k)},$
\end{enumerate}
where for all $ e\tilde{\in}A,~ x\tilde{\in}U,~ \gamma_1^{\sigma(k)} $and $\gamma_2^{\sigma(k)}  $ stand for the kth largest interval number in the IVHFEs $~ \mu_{F(e)(x)} $ and $\mu_{ G(e)(x)}, $ respectively. In this case, we write $ F_A\tilde{\subseteq}G_A. $
\end{defn}
\begin{defn}
	\cite{zhang} The complement of $ F_A, $ denoted by $ F_A^C, $ is defined by $F_A^C(e)=\{<h, \mu_{\tilde{F^C(e)(x)}}> : h\in U\} , $ where $ \mu_{F^C}:A\rightarrow IVHF(U) $is a mapping given by $\mu_{\tilde{F^C(e)}}, \forall e\tilde{\in}A $ such that $\mu_{\tilde{F^C(e)}} $ is the complement of interval-valued hesitant fuzzy element $ \mu_{\tilde{F(e)}}  $ on $ U. $
\end{defn}
\begin{defn}
	\cite{zhang} An interval-valued hesitant fuzzy soft set is said to be an empty interval-valued hesitant fuzzy soft set, denoted by $ \tilde{\phi},$ if $ F:E\rightarrow IVHF(U)$ such that $\tilde{F(e)}=\{<h, \mu_{\tilde{F(e)(x)}}> : h\in U\}=\{<h, \{[0,0]\}> : h\in U\}, \forall e\tilde{\in}E. $ 
\end{defn}

\begin{defn}
	\cite{zhang} An interval-valued hesitant fuzzy soft set is said to be an full interval-valued hesitant fuzzy soft set, denoted by $ \tilde{E},$ if $ F:E\rightarrow IVHF(U)$ such that\\ $\tilde{F(e)}=\{<h, \mu_{\tilde{F(e)(x)}}> : h\in U\}=\{<h, \{[1,1]\}> : h\in U\}, \forall e\tilde{\in}E. $ 
\end{defn}
\begin{defn}
	\cite{zhang} The ring sum operation on the two  interval-valued hesitant fuzzy soft sets $F_A, G_B  $ over $ (U,E),  $ denoted by $ F_A\oplus G_A=H, $ is a mapping given by 	 $ H:E\rightarrow IVHF(U) $
	  such that\\ 
	  $ \forall e\tilde{\in}E $ 
	  $\tilde{H(e)}=\{<h, \mu_{\tilde{H(e)(x)}}> : h\in U\}=\{<h, \mu_{\tilde{H(e)(x)}}\oplus \mu_{\tilde{G(e)(x)}}> : h\in U\}, \forall e\tilde{\in}E. $ 
\end{defn}
\begin{defn}
	\cite{zhang} The ring product operation on the two  interval-valued hesitant fuzzy soft sets $F_A, G_B  $ over $ (U,E),  $ denoted by $ F_A\otimes G_A=H, $ is a mapping given by
	$ H:E\rightarrow IVHF(U)$
	such that\\ $ \forall e\tilde{\in}E $
	$\tilde{H(e)}=\{<h, \mu_{\tilde{H(e)(x)}}> : h\in U\}=\{<h, \mu_{\tilde{H(e)(x)}}\otimes \mu_{\tilde{G(e)(x)}}> : h\in U\},\forall e\tilde{\in}E. $ 
\end{defn}
\section{Main Results }

\begin{defn}
The union of two interval-valued hesitant fuzzy soft sets $ F_A $ and $ G_B $ over $ (U,E) ,$ is the interval-valued hesitant fuzzy soft set $ H_C ,$ where $ C=A\cup B $ and  $\forall e\tilde{\in}C,$\\

$\mu_{H(e)}=\begin{cases}
	\mu_{F(e)}, & \textrm{if $e\tilde{\in}A-B;$}\\
	\mu_{G(e)}, & \textrm{if $ e\tilde{\in}A-B;$}\\
	\mu_{F(e)}\cup \mu_{G(e)}, & \textrm{if $e\tilde{\in}A\cap B.$}\end{cases}$
	
	We write $ F_A \tilde{\cup}G_B=H_C.$
\end{defn}
\begin{problem}
	Let  $F_A=\{e_1=\{<h_1,[0.3, 0.8]>, <h_2, [0.3, 0.8],[0.5,0.6],[0.3,0.6]>\}\\  e_2=\{<h_1,[0.2, 0.9],[0.7,1.0] >, <h_2, [0.8,1.0],[0.2,0.6]>\}.$\\
	
	$G_B=\{e_1=\{<h_1,[0.7, 0.9],[0.0,0.6]>, <h_2, [0.4, 0.7],[0.4,0.5]>\}\\  e_2=\{<h_1,[0.6, 0.8]>, <h_2, [0.3,0.8],[0.3,0.6]>\}\\
 e_3=\{<h_1,[0.5, 0.6],[0.3,0.6] >, <h_2, [0.1,0.6],[0.3,0.9],[0.3,0.6]>\}.$\\
 
  Now rearrange the membership value of	$ F_A$ and $ G_B$  with the help of  Definitions 2.9 , 2.10 and assumptions given by \cite{chen}, we have\\
  $F_A=\{e_1=\{<h_1,[0.3, 0.8]>, <h_2,[0.3,0.6], [0.3, 0.8],[0.5,0.6]>\}\\  e_2=\{<h_1,[0.2, 0.9],[0.7,1.0] >, <h_2, [0.2,0.6],[0.8,1.0]>\}.$\\
  
  	$G_B=\{e_1=\{<h_1,[0.0,0.6],[0.7, 0.9]>, <h_2, [0.4,0.5],[0.4, 0.7],[0.4,0.7]>\}\\  e_2=\{<h_1,[0.6, 0.8]>, <h_2, [0.3,0.6],[0.3,0.8]>\}\\
  	e_3=\{<h_1,[0.3,0.6] ,[0.5, 0.6]>, <h_2, [0.1,0.6],[0.3,0.9],[0.3,0.6]>\}.$\\
  	
	Therefore \\
	$ F_A \tilde{\cup}G_B= H_{A\tilde{\cup}B}=H_C\\
	= \{e_1=\{<h_1,[0.3, 0.8],[0.7,0.9]>, <h_2, [0.4, 0.6],[0.4,0.8][0.5,0.7]>\}\\  e_2=\{<h_1,[0.6, 0.9][0.7,1.0]>, <h_2, [0.3,0.6],[0.8,1.0]>\}\\
	e_3=\{<h_1,[0.3, 0.6],[0.5,0.6] >, <h_2, [0.1,0.6],[0.3,0.9],[0.3,0.6]>\}.$
	
\end{problem}

\begin{defn}
 The intersection of two interval-valued hesitant fuzzy soft sets $ F_A $ and $ G_B $ with $ A\cap B\neq\phi $ over $ (U,E) ,$ is the interval-valued hesitant fuzzy soft set $ H_C ,$ where $ C=A\cap B, $ and  $\forall e\tilde{\in}C, \mu_{ H(e)}=\mu_{F(e)}\cap \mu_{G(e)}.$ 
	We write $ F_A \tilde{\cap}G_B=H_C.$	

\end{defn}

\begin{problem} From Example 3.2, we have\\
	$ F_A \tilde{\cap}G_B= H_{A\tilde{\cap}B}=H_C$\\
$	= \{e_1=\{<h_1,[0.0, 0.6],[0.3,0.8]>, <h_2, [0.3, 0.5],[0.3,0.7][0.4,0.6]>\}\\  e_2=\{<h_1,[0.2, 0.8][0.6,0.8]>, <h_2, [0.2,0.6],[0.3,0.8]>\}.$
\end{problem}
\begin{prop}\label{prop38}
Let 	$ F_A$ be a interval-valued hesitant fuzzy soft set. Then the following are true:
\begin{enumerate}
\item[(i)]  $F_A \tilde{\cup}F_A=F_A$
\item[(ii)]  $F_A \tilde{\cap}F_A=F_A$
\item[(iii)]  $F_A \tilde{\cup}\tilde{\phi_A}=F_A$
\item[(iv)]   $F_A \tilde{\cap}\tilde{\phi_A}=\tilde{\phi_A}$
\item[(v)]  $F_A \tilde{\cup}\tilde{E_A}=\tilde{E_A}$
\item[(vi)] $F_A \tilde{\cap}\tilde{E_A}=F_A.$
\end{enumerate}
\end{prop}
\begin{proof}
 Obvious. 
\end{proof}

\begin{prop}\label{prop38}
Let $ F_A $ and $ G_A $ are two interval-valued hesitant fuzzy soft sets. Then 
\begin{enumerate}
\item[(i)] $(F_A\tilde{\cup}G_A)^C=F_A^C\tilde{\cap}G_A^C$
\item[(ii)] $(F_A\tilde{\cap}G_A)^C=F_A^C\tilde{\cup}G_A^C.$
\end{enumerate}
\end{prop}

\begin{proof}
(i) Let $ F_A^C \tilde{\cap}G_A^C=H_A.$\\
We have $ \forall e\tilde{\in} A,\mu_{ H(e)}=\mu_{F^C(e)}\cap \mu_{G^C(e)}. $	..................(A1)\\
Suppose that $ F_A\tilde{\cup}G_A=L_A $\\
Therefore, $ (F_A\tilde{\cup}G_A)^C=L^C_A.$\\
We have $ \forall e\tilde{\in} A,\mu_{ L^C(e)}=(\mu_{F(e)}\cup \mu_{G(e)})^C=\mu_{F^C(e)}\cap \mu_{G^C(e)}. $	..................(A2)\\

From (A1) and (A2), $(F_A\tilde{\cup}G_A)^C=F_A^C\tilde{\cap}G_A^C.$\\

(ii) Let $ F_A^C \tilde{\cup}G_A^C=P_A.$\\
We have $ \forall e\tilde{\in} A,\mu_{ P(e)}=\mu_{F^C(e)}\cup \mu_{G^C(e)}. $	..................(B1)\\
Suppose that $ F_A\tilde{\cap}G_A=Q_A $\\
Therefore $ (F_A\tilde{\cap}G_A)^C=Q^C_A.$\\
We have $ \forall e\tilde{\in} A,\mu_{ Q^C(e)}=(\mu_{F(e)}\cap \mu_{G(e)})^C=\mu_{F^C(e)}\cup \mu_{G^C(e)}. $	..................(B2)\\

From (B1) and (B2), $(F_A\tilde{\cap}G_A)^C=F_A^C\tilde{\cup}G_A^C.$
\end{proof}

\begin{prop}\label{prop38}
	Let 	$ F_A$ and $ G_B $ are two interval-valued hesitant fuzzy soft sets. Then the following  are satisfied:
\begin{enumerate}
	\item[(i)] $F_A^C\tilde{\cap}G_B^C\tilde{\subseteq}(F_A\tilde{\cup}G_B)^C$
	\item[(ii)] $(F_A\tilde{\cap}G_B)^C\tilde{\subseteq}F_A^C\tilde{\cup}G_B^C$
	\item[(iii)]$ F_A^C\tilde{\cap}G_B^C\tilde{\subseteq}(F_A\tilde{\cap}G_B)^C $
	\item[(iv)] $ (F_A\tilde{\cup}G_B)^C\tilde{\subseteq}F_A^C\tilde{\cup}G_B^C.$
\end{enumerate}
\end{prop}
\begin{proof}
		From Example 3.2\\
	(i)  $(F_A\tilde{\cup}G_B)^C= \{e_1=\{<h_1,[0.1, 0.3],[0.2,0.7]>, <h_2, [0.3, 0.5],[0.2,0.6][0.4,0.6]>\}\\  e_2=\{<h_1,[0.0, 0.3][0.1,0.4]>, <h_2, [0.0,0.2],[0.4,0.7]>\}\\
	 e_3=\{<h_1,[0.4, 0.5],[0.4,0.7] >, <h_2, [0.4,0.7],[0.1,0.7],[0.4,0.9]>\}.$\\

	 $F_A^C=\{e_1=\{<h_1,[0.2, 0.7]>, <h_2, [0.4, 0.5],[0.2,0.7],[0.4,0.7]>\}\\  e_2=\{<h_1,[0.0, 0.3],[0.1,0.8] >, <h_2, [0.0,0.2],[0.4,0.8]>\}.$\\
	 
	 $G_B^C=\{e_1=\{<h_1,[0.1, 0.3],[0.4,1.0]>, <h_2,[0.3,0.6],[0.3,0.6],[0.5,0.6]>\}\\  e_2=\{<h_1,[0.2, 0.4]>, <h_2, [0.2,0.7],[0.4,0.7]>\}\\
	 e_3=\{<h_1,[0.4, 0.5],[0.4,0.7] >, <h_2, [0.4,0.7],[0.1,0.7],[0.4,0.9]>\}.$\\ 
	 
		$ F_A^C\tilde{\cap}G_B^C= \{e_1=\{<h_1,[0.1, 0.3],[0.2,0.7]>, <h_2, [0.3, 0.5],[0.2,0.6][0.4,0.6]>\}\\  e_2=\{<h_1,[0.0, 0.3][0.1,0.4]>, <h_2, [0.0,0.2],[0.4,0.7]>\}.$\\ 
		
		Hence $F_A^C\tilde{\cap}G_B^C\tilde{\subseteq}(F_A\tilde{\cup}G_B)^C.$\\
	
(ii) 		$ (F_A \tilde{\cap}G_B)^C
= \{e_1=\{<h_1,[0.2, 0.7],[0.4,1.0]>, <h_2, [0.4, 0.6],[0.3,0.7][0.5,0.7]>\}\\  e_2=\{<h_1,[0.2, 0.4][0.2,0.8]>, <h_2, [0.2,0.7],[0.4,0.8]>\}.$\\
	
$ F_A^C \tilde{\cup}G_B^C
= \{e_1=\{<h_1,[0.2, 0.7],[0.4,1.0]>, <h_2, [0.4, 0.6],[0.3,0.7][0.5,0.7]>\}\\  e_2=\{<h_1,[0.2, 0.4][0.2,0.8]>, <h_2, [0.2,0.7],[0.4,0.8]>\}\\  e_3=\{<h_1,[0.4, 0.5],[0.4,0.7] >, <h_2, [0.4,0.7],[0.1,0.7],[0.4,0.9]>\}.$\\
Hence  $(F_A\tilde{\cap}G_B)^C\tilde{\subseteq}F_A^C\tilde{\cup}G_B^C.$\\

(iii) From (i) and (ii) we get the result.\\
 
(iv) From (i) and (ii) we get the result. 
	\end{proof}

\begin{prop}\label{prop38}
	Let 	$ F_A, G_B$ and $ H_C $ are three interval-valued hesitant fuzzy soft sets. Then the following are satisfied:
	
	\begin{enumerate}
		\item[(i)] $F_A\tilde{\cup}G_B=G_B\tilde{\cup}F_A$
		\item[(ii)]$F_A\tilde{\cap}G_B=G_B\tilde{\cap}F_A$
		\item[(iii)] $F_A\tilde{\cup}(G_B\tilde{\cup}H_C)=(F_A\tilde{\cup} G_B)\tilde{\cup}H_C$
		
		\item[(iv)] $F_A\tilde{\cap}(G_B\tilde{\cap}H_C)=(F_A\tilde{\cap} G_B)\tilde{\cap}H_C.$
	\end{enumerate}
\end{prop}
\begin{proof}
	
	The proof can be obtained from definition 3.1 and definition 3.3.
\end{proof}

\begin{prop}\label{prop38}
	Let 	$ F_A, G_A$ and $ H_A $ are three interval-valued hesitant fuzzy soft sets. Then the following propositiones are satiesfied:
	
	\begin{enumerate}
		\item[(i)] $F_A\tilde{\cup}G_A=G_A\tilde{\cup}F_A$
		\item[(ii)]$F_A\tilde{\cap}G_A=G_A\tilde{\cap}F_A$
		\item[(iii)] $F_A\tilde{\cup}(G_A\tilde{\cup}H_A)=(F_A\tilde{\cup} G_A)\tilde{\cup}H_A$
		
		\item[(iv)] $F_A\tilde{\cap}(G_A\tilde{\cap}H_A)=(F_A\tilde{\cap} G_A)\tilde{\cap}H_A.$
	\end{enumerate}
\end{prop}
\begin{proof}
The proof can be obtained from definition 3.1 and definition 3.3.
\end{proof}

\begin{prop}\label{prop38}
	Let 	$ F_A, G_A$ and $ H_A $ are three interval-valued hesitant fuzzy soft sets. Then the following  are satisfied:
	
	\begin{enumerate}
		\item[(i)] $F_A\tilde{\cup}(G_A\tilde{\cap}H_A)=(F_A\tilde{\cup} G_A)\tilde{\cap}(F_A\tilde{\cup}H_A)$
		\item[(ii)]$F_A\tilde{\cap}(G_A\tilde{\cup}H_A)=(F_A\tilde{\cap} G_A)\tilde{\cup}(F_A\tilde{\cap}H_A).$

	\end{enumerate}
\end{prop}
\begin{proof}
	Obvious. 
\end{proof}
\begin{prop}\label{prop38}
	Let 	$ F_A, G_B$ and $ H_C $ are three interval-valued hesitant fuzzy soft sets. Then the following  are not satisfied:

	\begin{enumerate}
		\item[(i)] $F_A\tilde{\cup}(G_B\tilde{\cap}H_C)=(F_A\tilde{\cup} G_B)\tilde{\cap}(F_A\tilde{\cup}H_C)$
		\item[(ii)]$F_A\tilde{\cap}(G_B\tilde{\cup}H_C)=(F_A\tilde{\cap} G_B)\tilde{\cup}(F_A\tilde{\cap}H_C).$
		
	\end{enumerate}
\end{prop}

\begin{proof}
We consider a example.\\
Let  $H_C=\{e_2=\{<h_1,[0.4, 0.6],[o.2,0.6],[0.7,1.0]>, <h_2,[0.3,0.8],>\}\\  e_3=\{<h_1,[0.2, 0.5],[0.3,0.5] >, <h_2, [0.6,0.8],[0.2,0.5]>\}.$\\
 Now by Definitions 2.9 and 2.10 and assumptions given by \cite{chen}\\
$H_C=\{e_2=\{<h_1,[0.2, 0.6],[o.4,0.6],[0.7,1.0]>, <h_2,[0.3,0.8],>\}\\  e_3=\{<h_1,[0.2, 0.5],[0.3,0.5] >, <h_2, [0.2,0.5],[0.6,0.8]>\}.$\\
(i) From example 3.2, we have\\
 $F_A\tilde{\cup}H_C= \{e_1=\{<h_1,[0.3, 0.8]>, <h_2, [0.3, 0.6],[0.3,0.8][0.5,0.6]>\}\\  e_2=\{<h_1,[0.2, 0.9],[0.7,1.0],[0.7,1.0]>, <h_2, [0.3,0.8],[0.8,1.0]>\}\\
 e_3=\{<h_1,[0.2, 0.5],[0.3,0.5] >, <h_2, [0.2,0.5],[0.6,0.8]>\}.$\\
 $(F_A\tilde{\cup} G_B)\tilde{\cap}(F_A\tilde{\cup}H_C)\\
 =\{e_1=\{<h_1,[0.3, 0.8],[0.3, 0.8]>, <h_2, [0.3, 0.6],[0.3,0.8][0.5,0.6]>\}\\  e_2=\{<h_1,[0.2, 0.9],[0.7,1.0],[0.7,1.0]>, <h_2, [0.3,0.6],[0.8,1.0]>\}\\
 e_3=\{<h_1,[0.2, 0.5],[0.3,0.5] >, <h_2, [0.1,0.5],[0.3,0.8],[0.3, 0.6]>\}.$\\
 Again\\
  $G_B\tilde{\cap}H_C= \{e_2=\{<h_1,[0.2, 0.6],[0.4,0.6],[0.6,0.8]>, <h_2, [0.3,0.6],[0.3,0.8]>\}\\
 e_3=\{<h_1,[0.2, 0.5],[0.3,0.5] >, <h_2, [0.1,0.5],[0.3,0.8],[0.3,0.6]>\}.$\\
 Therefore\\ $F_A\tilde{\cup}(G_B\tilde{\cap}H_C)$\\
$ =\{e_1=\{<h_1,[0.3, 0.8]>, <h_2, [0.3, 0.6],[0.3,0.8][0.5,0.6]>\}\\  e_2=\{<h_1,[0.2, 0.9],[0.7,1.0],[0.7,1.0]>, <h_2, [0.3,0.6],[0.8,1.0]>\}\\
 e_3=\{<h_1,[0.2, 0.5],[0.3,0.5] >, <h_2, [0.1,0.5],[0.3,0.8],[0.3, 0.6]>\}.$\\
 Hence $F_A\tilde{\cup}(G_B\tilde{\cap}H_C) \neq(F_A\tilde{\cup} G_B)\tilde{\cap}(F_A\tilde{\cup}H_C).$\\
 
 (ii) From example 3.2 and 3.4.\\
 
  $G_B\tilde{\cup}H_C= \{e_1=\{<h_1,[0.0, 0.6],[0.7,0.9]>, <h_2, [0.4, 0.5],[0.4,0.7][0.4,0.7]>\}\\  e_2=\{<h_1,[0.6, 0.8],[0.6,0.8],[0.7,1.0]>, <h_2, [0.3,0.8],[0.3,0.8]>\}\\
  e_3=\{<h_1,[0.3, 0.6],[0.5,0.6] >, <h_2, [0.2,0.6],[0.6,0.9],[0.6,0.8]>\}.$\\
 Therefore\\ $F_A\tilde{\cap}(G_B\tilde{\cup}H_C)$\\
$ =\{e_1=\{<h_1,[0.0, 0.6],[0.3,0.8]>, <h_2, [0.3, 0.5],[0.3,0.7][0.4,0.6]>\}\\  e_2=\{<h_1,[0.2, 0.8],[0.6,0.8],[0.7,1.0]>, <h_2, [0.2,0.6],[0.3,0.8]>\}.$\\
 
Again\\
  $F_A\tilde{\cap}H_C= \{e_2=\{<h_1,[0.2, 0.6],[0.4,0.6],[0.7,1.0]>, <h_2, [0.2,0.6],[0.3,0.8]>\}.$\\ 
Therefore \\
 $(F_A\tilde{\cap} G_B)\tilde{\cup}(F_A\tilde{\cap}H_C)$\\
$=\{e_2=\{<h_1,[0.2, 0.8],[0.6,0.8],[0.7,1.0]>, <h_2, [0.2,0.6],[0.3,0.8]>\}.$\\ 
 
Hence $F_A\tilde{\cap}(G_B\tilde{\cup}H_C) \neq(F_A\tilde{\cap} G_B)\tilde{\cup}(F_A\tilde{\cap}H_C).$ 
\end{proof}
\begin{defn}
Let $ \Re=\{(F_i)_{A_i}: i\tilde{\in}I\} $ be a family of hesitant fuzzy soft sets over $ (U,E) .$Then the union of hesitant fuzzy soft sets in $ \Re $ is a hesitant fuzzy soft set $ H_K, K=\cup_iA_i $ and $ \forall e\tilde{\in}E, $ $ K(e)=\cup_i(\vartriangle_i)_{A_i}(e) ,$
where\\ $ (\vartriangle_i)_{A_i}(e) =\begin{cases}
F_i(e), & \textrm{if $e\tilde{\in}A_i$}\\
\phi, & \textrm{if $ e\tilde{\notin}A_i.$}\end{cases} $
\end{defn}
\begin{problem}
	Let  $(F_1)_{A_1}=\{e_1=\{<h_1,[0.3, 0.8]>, <h_2, [0.3, 0.8],[0.5,0.6],[0.3,0.6]>\}\\  e_2=\{<h_1,[0.2, 0.9],[0.7,1.0] >, <h_2, [0.8,1.0],[0.2,0.6]>\}.$\\
	
	$(F_2)_{A_2}=\{e_1=\{<h_1,[0.7, 0.9],[0.0,0.6]>, <h_2, [0.4, 0.7],[0.4,0.5]>\}\\  e_2=\{<h_1,[0.6, 0.8]>, <h_2, [0.3,0.8],[0.3,0.6]>\}\\
	e_3=\{<h_1,[0.5, 0.6],[0.3,0.6] >, <h_2, [0.1,0.6],[0.3,0.9],[0.3,0.6]>\}.$\\
 $(F_3)_{A_3}=\{e_2=\{<h_1,[0.4, 0.6],[o.2,0.6],[0.7,1.0]>, <h_2,[0.3,0.8],>\}\\  e_3=\{<h_1,[0.2, 0.5],[0.3,0.5] >, <h_2, [0.6,0.8],[0.2,0.5]>\}.$\\
 Therefore\\ 
 $(F_1)_{A_1}\tilde{\cup}(F_2)_{A_2}\tilde{\cup}(F_3)_{A_3}$\\
 $= \{e_1=\{<h_1,[0.3, 0.8],[0.7,0.9]>,
 <h_2, [0.4, 0.6],[0.4,0.8],[0.5,0.7]>\}\\  e_2=\{<h_1,[0.6, 0.9],[0.7,1.0],[0.7,1.0] >, <h_2, [0.3,0.8],[0.8,1.0]>\}\\
 e_3=\{<h_1,[0.3, 0.6],[0.5,0.6] >, <h_2, [0.2,0.6],[0.6,0.9],[0.6,0.8]>\}.$
\end{problem}
\begin{defn}
	Let $ \Re=\{(F_i)_{A_i}: i\tilde{\in}I\} $ be a family of hesitant fuzzy soft sets with $ \cap_i{A_i} \neq\phi$ over $ (U,E) .$Then the intersection of hesitant fuzzy soft sets in $ \Re $ is a hesitant fuzzy soft set $ H_K, K=\cap_iA_i $ and $ \forall e\tilde{\in}E, $
	$ K(e)=\cap_i{A_i}(e).$
\end{defn}
\begin{problem}
From Example 3.13, we have\\
 $ (F_1)_{A_1}\tilde{\cap}(F_2)_{A_2}\tilde{\cap}(F_3)_{A_3}$\\
 $= \{ e_2=\{<h_1,[0.2, 0.6],[0.4,0.6],[0.6,0.8] >, <h_2, [0.2,0.6],[0.3,0.8]>\}.$
\end{problem}
\begin{prop}
Let $ \Re=\{(F_i)_{A_i}: i\tilde{\in}I\} $ be a family of hesitant fuzzy soft sets over $ (U,E) .$ Then
\begin{enumerate}
	\item[(i)]$ \tilde{\bigcap}_i{(F_i)_{A_i}}^C\tilde{\subseteq}(\tilde{\bigcup}_i{(F_i)_{A_i}})^C $
	\item[(ii)] $ (\tilde{\bigcap}_i{(F_i)_{A_i}})^C\tilde{\subseteq}\tilde{\bigcup}_i{(F_i)_{A_i}}^C.$
\end{enumerate}
\end{prop}
\begin{proof}
	
	Obvious. 
\end{proof}

\begin{prop}
	Let $ \Re=\{(F_i)_A: i\tilde{\in}I\} $ be a family of hesitant fuzzy soft sets over $ (U,E) .$ Then 
	\begin{enumerate}
		\item[(i)]$ \tilde{\bigcap}_i{(F_i)_A}^C=(\tilde{\bigcup}_i{(F_i)_A})^C $
		\item[(ii)] $ (\tilde{\bigcap}_i{(F_i)_A})^C=\tilde{\bigcup}_i{(F_i)_A}^C.$
	\end{enumerate}
\end{prop}
\begin{proof}
	
	Obvious. 
\end{proof}

\section{New operators on interval-valued hesitant fuzzy soft elements}

\begin{defn}
Let $ \tilde{\mu_1},  \tilde{\mu_2} $ be two interval-valued hesitant fuzzy soft elements (IVHFSEs) of same set of parameters, then 
\begin{enumerate}
\item[(i)] $\tilde{\mu_1} O_1 \tilde{\mu_2}=\bigcup_{\gamma_1\tilde{\in}\mu_1,\gamma_2\tilde{\in}\mu_2 }[\frac{\mid{\gamma_1 }^L-{\gamma_2 }^L\mid}{1+\mid{\gamma_1 }^L-{\gamma_2 }^L\mid},\frac{\mid{\gamma_1 }^U-{\gamma_2 }^U\mid}{1+\mid{\gamma_1 }^U-{\gamma_2 }^U\mid} ]$
	
\item[(ii)] $\tilde{\mu_1} O_2 \tilde{\mu_2}=\bigcup_{\gamma_1\tilde{\in}\mu_1,\gamma_2\tilde{\in}\mu_2 }[\frac{\mid{\gamma_1 }^L-{\gamma_2 }^L\mid}{1+2\mid{\gamma_1 }^L-{\gamma_2 }^L\mid},\frac{\mid{\gamma_1 }^U-{\gamma_2 }^U\mid}{1+2\mid{\gamma_1 }^U-{\gamma_2 }^U\mid} ]$	
	
\item[(iii)] $\tilde{\mu_1} O_3 \tilde{\mu_2}=\bigcup_{\gamma_1\tilde{\in}\mu_1,\gamma_2\tilde{\in}\mu_2 }[\frac{\mid{\gamma_1 }^L-{\gamma_2 }^L\mid}{2},\frac{\mid{\gamma_1 }^U-{\gamma_2 }^U\mid}{2} ]$	
\item[(iv)] $\tilde{\mu_1} O_4 \tilde{\mu_2}=\bigcup_{\gamma_1\tilde{\in}\mu_1,\gamma_2\tilde{\in}\mu_2 }[\frac{\mid{\gamma_1 }^L\ast{\gamma_2 }^L\mid}{2},\frac{\mid{\gamma_1 }^U\ast{\gamma_2 }^U\mid}{2} ].$	
\end{enumerate}
\end{defn}
\begin{prop}If $ \tilde{\mu_1},  \tilde{\mu_2} $ and $ \tilde{\mu_3} $ be two interval-valued hesitant fuzzy soft elements. Then the following identites are true: 
\begin{enumerate}
\item[(i)] $(\tilde{\mu_1}\oplus\tilde{\mu_2})\tilde{\cap}(\tilde{\mu_1} O_1 \tilde{\mu_2})=\tilde{\mu_1} O_1 \tilde{\mu_2}  ,$
\item[(ii)] $(\tilde{\mu_1}\oplus\tilde{\mu_2})\tilde{\cup}(\tilde{\mu_1} O_1 \tilde{\mu_2})=\tilde{\mu_1} \oplus \tilde{\mu_2}  ,$
\item[(iii)] $(\tilde{\mu_1}\otimes\tilde{\mu_2})\tilde{\cap}(\tilde{\mu_1} O_1 \tilde{\mu_2})=\tilde{\mu_1} O_1 \tilde{\mu_2}  ,$	
\item[(iv)] $(\tilde{\mu_1}\otimes\tilde{\mu_2})\tilde{\cup}(\tilde{\mu_1} O_1 \tilde{\mu_2})=\tilde{\mu_1} \otimes \tilde{\mu_2}  ,$
\item[(v)] $(\tilde{\mu_1}\tilde{\cup}\tilde{\mu_2})O_1\tilde{\mu_3}=(\tilde{\mu_1} O_1 \tilde{\mu_3})\tilde{\cup}(\tilde{\mu_2} O_1 \tilde{\mu_3}), $
\item[(vi)] $(\tilde{\mu_1}\tilde{\cap}\tilde{\mu_2})O_1\tilde{\mu_3}=(\tilde{\mu_1} O_1 \tilde{\mu_3})\tilde{\cap}(\tilde{\mu_2} O_1 \tilde{\mu_3}), $		
\end{enumerate}
\end{prop}
\begin{proof}
		(i) $(\tilde{\mu_1}\oplus\tilde{\mu_2})\tilde{\cap}(\tilde{\mu_1} O_1 \tilde{\mu_2})\\
		= (\bigcup_{\gamma_1\tilde{\in}\mu_1,\gamma_2\tilde{\in}\mu_2 }[\tilde{\gamma_1}^L+\tilde{\gamma_2}^L-\tilde{\gamma_1}^L.\tilde{\gamma_2}^L, \tilde{\gamma_1}^U+\tilde{\gamma_2}^U-\tilde{\gamma_1}^U.\tilde{\gamma_2}^U])\tilde{\cap}(\bigcup_{\gamma_1\tilde{\in}\mu_1,\gamma_2\tilde{\in}\mu_2 }[\frac{\mid{\gamma_1 }^L-{\gamma_2 }^L\mid}{1+\mid{\gamma_1 }^L-{\gamma_2 }^L\mid},\frac{\mid{\gamma_1 }^U-{\gamma_2 }^U\mid}{1+\mid{\gamma_1 }^U-{\gamma_2 }^U\mid} ])\\
		= \bigcup_{\gamma_1\tilde{\in}\mu_1,\gamma_2\tilde{\in}\mu_2 }[\textrm{min}\{\tilde{\gamma_1}^L+\tilde{\gamma_2}^L-\tilde{\gamma_1}^L.\tilde{\gamma_2}^L,\frac{\mid{\gamma_1 }^L-{\gamma_2 }^L\mid}{1+\mid{\gamma_1 }^L-{\gamma_2 }^L\mid} \},\textrm{min}\{\tilde{\gamma_1}^U+\tilde{\gamma_2}^U-\tilde{\gamma_1}^U.\tilde{\gamma_2}^U,\frac{\mid{\gamma_1 }^U-{\gamma_2 }^U\mid}{1+\mid{\gamma_1 }^U-{\gamma_2 }^U\mid} \} ]\\
		=\bigcup_{\gamma_1\tilde{\in}\mu_1,\gamma_2\tilde{\in}\mu_2 }[\frac{\mid{\gamma_1 }^L-{\gamma_2 }^L\mid}{1+\mid{\gamma_1 }^L-{\gamma_2 }^L\mid},\frac{\mid{\gamma_1 }^U-{\gamma_2 }^U\mid}{1+\mid{\gamma_1 }^U-{\gamma_2 }^U\mid}]\\
		=\tilde{\mu_1} O_1 \tilde{\mu_2}.$	\\
		
		(ii)	$(\tilde{\mu_1}\oplus\tilde{\mu_2})\tilde{\cup}(\tilde{\mu_1} O_1 \tilde{\mu_2})\\
		=(\bigcup_{\gamma_1\tilde{\in}\mu_1,\gamma_2\tilde{\in}\mu_2 }[\tilde{\gamma_1}^L+\tilde{\gamma_2}^L-\tilde{\gamma_1}^L.\tilde{\gamma_2}^L, \tilde{\gamma_1}^U+\tilde{\gamma_2}^U-\tilde{\gamma_1}^U.\tilde{\gamma_2}^U])\tilde{\cup}(\bigcup_{\gamma_1\tilde{\in}\mu_1,\gamma_2\tilde{\in}\mu_2 }[\frac{\mid{\gamma_1 }^L-{\gamma_2 }^L\mid}{1+\mid{\gamma_1 }^L-{\gamma_2 }^L\mid},\frac{\mid{\gamma_1 }^U-{\gamma_2 }^U\mid}{1+\mid{\gamma_1 }^U-{\gamma_2 }^U\mid} ])\\
		= \bigcup_{\gamma_1\tilde{\in}\mu_1,\gamma_2\tilde{\in}\mu_2 }[\textrm{max}\{\tilde{\gamma_1}^L+\tilde{\gamma_2}^L-\tilde{\gamma_1}^L.\tilde{\gamma_2}^L,\frac{\mid{\gamma_1 }^L-{\gamma_2 }^L\mid}{1+\mid{\gamma_1 }^L-{\gamma_2 }^L\mid} \},\textrm{max}\{\tilde{\gamma_1}^U+\tilde{\gamma_2}^U-\tilde{\gamma_1}^U.\tilde{\gamma_2}^U,\frac{\mid{\gamma_1 }^U-{\gamma_2 }^U\mid}{1+\mid{\gamma_1 }^U-{\gamma_2 }^U\mid} \} ]\\
		=\bigcup_{\gamma_1\tilde{\in}\mu_1,\gamma_2\tilde{\in}\mu_2 }[\tilde{\gamma_1}^L+\tilde{\gamma_2}^L-\tilde{\gamma_1}^L.\tilde{\gamma_2}^L, \tilde{\gamma_1}^U+\tilde{\gamma_2}^U-\tilde{\gamma_1}^U.\tilde{\gamma_2}^U]\\
		=\tilde{\mu_1} \oplus \tilde{\mu_2} .$\\
			  
		(iii) $(\tilde{\mu_1}\otimes\tilde{\mu_2})\tilde{\cap}(\tilde{\mu_1} O_1 \tilde{\mu_2})\\
		= (\bigcup_{\gamma_1\tilde{\in}\mu_1,\gamma_2\tilde{\in}\mu_2 }[\tilde{\gamma_1}^L.\tilde{\gamma_2}^L, \tilde{\gamma_1}^U.\tilde{\gamma_2}^U])\tilde{\cap}(\bigcup_{\gamma_1\tilde{\in}\mu_1,\gamma_2\tilde{\in}\mu_2 }[\frac{\mid{\gamma_1 }^L-{\gamma_2 }^L\mid}{1+\mid{\gamma_1 }^L-{\gamma_2 }^L\mid},\frac{\mid{\gamma_1 }^U-{\gamma_2 }^U\mid}{1+\mid{\gamma_1 }^U-{\gamma_2 }^U\mid} ])\\
		= \bigcup_{\gamma_1\tilde{\in}\mu_1,\gamma_2\tilde{\in}\mu_2 }[\textrm{min}\{\tilde{\gamma_1}^L.\tilde{\gamma_2}^L,\frac{\mid{\gamma_1 }^L-{\gamma_2 }^L\mid}{1+\mid{\gamma_1 }^L-{\gamma_2 }^L\mid} \},\textrm{min}\{\tilde{\gamma_1}^U.\tilde{\gamma_2}^U,\frac{\mid{\gamma_1 }^U-{\gamma_2 }^U\mid}{1+\mid{\gamma_1 }^U-{\gamma_2 }^U\mid} \} ]\\
		=\bigcup_{\gamma_1\tilde{\in}\mu_1,\gamma_2\tilde{\in}\mu_2 }[\frac{\mid{\gamma_1 }^L-{\gamma_2 }^L\mid}{1+\mid{\gamma_1 }^L-{\gamma_2 }^L\mid},\frac{\mid{\gamma_1 }^U-{\gamma_2 }^U\mid}{1+\mid{\gamma_1 }^U-{\gamma_2 }^U\mid}]\\
		=\tilde{\mu_1} O_1 \tilde{\mu_2}. $\\
			
		(iv)	$ (\tilde{\mu_1}\otimes\tilde{\mu_2})\tilde{\cup}(\tilde{\mu_1} O_1 \tilde{\mu_2})\\
		=(\bigcup_{\gamma_1\tilde{\in}\mu_1,\gamma_2\tilde{\in}\mu_2 }[\tilde{\gamma_1}^L.\tilde{\gamma_2}^L, \tilde{\gamma_1}^U.\tilde{\gamma_2}^U])\tilde{\cup}(\bigcup_{\gamma_1\tilde{\in}h_1,\gamma_2\tilde{\in}h_2 }[\frac{\mid{\gamma_1 }^L-{\gamma_2 }^L\mid}{1+\mid{\gamma_1 }^L-{\gamma_2 }^L\mid},\frac{\mid{\gamma_1 }^U-{\gamma_2 }^U\mid}{1+\mid{\gamma_1 }^U-{\gamma_2 }^U\mid} ])\\
		= \bigcup_{\gamma_1\tilde{\in}\mu_1,\gamma_2\tilde{\in}\mu_2 }[\textrm{max}\{\tilde{\gamma_1}^L.\tilde{\gamma_2}^L,\frac{\mid{\gamma_1 }^L-{\gamma_2 }^L\mid}{1+\mid{\gamma_1 }^L-{\gamma_2 }^L\mid} \},\textrm{max}\{\tilde{\gamma_1}^U.\tilde{\gamma_2}^U,\frac{\mid{\gamma_1 }^U-{\gamma_2 }^U\mid}{1+\mid{\gamma_1 }^U-{\gamma_2 }^U\mid} \} ]\\
		=(\bigcup_{\gamma_1\tilde{\in}\mu_1,\gamma_2\tilde{\in}\mu_2 }[\tilde{\gamma_1}^L.\tilde{\gamma_2}^L, \tilde{\gamma_1}^U.\tilde{\gamma_2}^U])\\
		=\tilde{\mu_1} \otimes \tilde{\mu_2}. $\\
		
	(v) $(\tilde{\mu_1}\tilde{\cup}\tilde{\mu_2})O_1\tilde{\mu_3}\\
	=\bigcup_{\gamma_1\tilde{\in}\mu_1,\gamma_2\tilde{\in}\mu_2 }[\textrm{max}\{\tilde{\gamma_1}^L,\tilde{\gamma_2}^L\},\textrm{max}\{\tilde{\gamma_1}^U,\tilde{\gamma_2}^U \} ]O_1 \bigcup_{\gamma_3\tilde{\in}\mu_3}[\tilde{\gamma_3}^L, {\gamma_3}^U]\\
	=\bigcup_{\gamma_1\tilde{\in}\mu_1,\gamma_2\tilde{\in}\mu_2, \gamma_3\tilde{\in}\mu_3 }[\frac{\mid\textrm{max}\{{\gamma_1 }^L.{\gamma_2 }^L\}-{\gamma_3 }^L\mid}{1+\mid\textrm{max}\{{\gamma_1 }^L.{\gamma_2 }^L\}-{\gamma_3 }^L\mid},\frac{\mid\textrm{max}\{{\gamma_1 }^U.{\gamma_2 }^U\}-{\gamma_3 }^U\mid}{1+\mid\textrm{max}\{{\gamma_1 }^U.{\gamma_2 }^U\}-{\gamma_3 }^U\mid}]\\
	=\bigcup_{\gamma_1\tilde{\in}\mu_1,\gamma_2\tilde{\in}\mu_2,\gamma_3\tilde{\in}\mu_3  }[\textrm{max}\{\frac{\mid{\gamma_1 }^L-{\gamma_2 }^L\mid}{1+\mid{\gamma_1 }^L-{\gamma_2 }^L\mid},\frac{\mid{\gamma_2
		}^L-{\gamma_3 }^L\mid}{1+\mid{\gamma_2 }^L-{\gamma_3 }^L\mid}\}, \textrm{max}\{\frac{\mid{\gamma_1 }^U-{\gamma_2 }^U\mid}{1+\mid{\gamma_1 }^U-{\gamma_2 }^U\mid},\frac{\mid{\gamma_2
	}^U-{\gamma_3 }^U\mid}{1+\mid{\gamma_2 }^U-{\gamma_3 }^U\mid}\} ]\\
=(\tilde{\mu_1} O_1 \tilde{\mu_3})\tilde{\cup}(\tilde{\mu_2} O_1 \tilde{\mu_3}). $\\		

(vi)	$(\tilde{\mu_1}\tilde{\cap}\tilde{\mu_2})O_1\tilde{\mu_3}\\
= \bigcup_{\gamma_1\tilde{\in}\mu_1,\gamma_2\tilde{\in}\mu_2 }[\textrm{min}\{\tilde{\gamma_1}^L,\tilde{\gamma_2}^L\},\textrm{min}\{\tilde{\gamma_1}^U,\tilde{\gamma_2}^U \} ]O_1 \bigcup_{\gamma_3\tilde{\in}h_3}[\tilde{\gamma_3}^L, {\gamma_3}^U]\\
=\bigcup_{\gamma_1\tilde{\in}\mu_1,\gamma_2\tilde{\in}\mu_2, \gamma_3\tilde{\in}\mu_3 }[\frac{\mid\textrm{min}\{{\gamma_1 }^L.{\gamma_2 }^L\}-{\gamma_3 }^L\mid}{1+\mid\textrm{min}\{{\gamma_1 }^L.{\gamma_2 }^L\}-{\gamma_3 }^L\mid},\frac{\mid\textrm{min}\{{\gamma_1 }^U.{\gamma_2 }^U\}-{\gamma_3 }^U\mid}{1+\mid\textrm{min}\{{\gamma_1 }^U.{\gamma_2 }^U\}-{\gamma_3 }^U\mid}]\\
=\bigcup_{\gamma_1\tilde{\in}\mu_1,\gamma_2\tilde{\in}\mu_2 }[\textrm{min}\{\frac{\mid{\gamma_1 }^L-{\gamma_2 }^L\mid}{1+\mid{\gamma_1 }^L-{\gamma_2 }^L\mid},\frac{\mid{\gamma_2
	}^L-{\gamma_3 }^L\mid}{1+\mid{\gamma_2 }^L-{\gamma_3 }^L\mid}\}, \textrm{min}\{\frac{\mid{\gamma_1 }^U-{\gamma_2 }^U\mid}{1+\mid{\gamma_1 }^U-{\gamma_2 }^U\mid},\frac{\mid{\gamma_2
}^U-{\gamma_3 }^U\mid}{1+\mid{\gamma_2 }^U-{\gamma_3 }^U\mid}\} ]\\
=(\tilde{\mu_1} O_1 \tilde{\mu_3})\tilde{\cap}(\tilde{\mu_2} O_1 \tilde{\mu_3}).$
\end{proof}			
\begin{prop}If $ \tilde{\mu_1},  \tilde{\mu_2} $ and $ \tilde{\mu_3} $ be two interval-valued hesitant fuzzy soft elements. Then the following identites are true: 
	\begin{enumerate}			
\item[(i)] $(\tilde{\mu_1}\oplus\tilde{\mu_2})\tilde{\cap}(\tilde{\mu_1} O_2 \tilde{\mu_2})=\tilde{\mu_1} O_2  \tilde{\mu_2}  ,$	
\item[(ii)] $(\tilde{\mu_1}\oplus\tilde{\mu_2})\tilde{\cup}(\tilde{\mu_1} O_2 \tilde{\mu_2})=\tilde{\mu_1} \oplus \tilde{\mu_2}  ,$	
\item[(iii)] $(\tilde{\mu_1}\otimes\tilde{\mu_2})\tilde{\cap}(\tilde{\mu_1} O_2 \tilde{\mu_2})=\tilde{\mu_1} O_2 \tilde{\mu_2}  ,$	
\item[(iv)] $(\tilde{\mu_1}\otimes\tilde{\mu_2})\tilde{\cup}(\tilde{\mu_1} O_2 \tilde{\mu_2})=\tilde{\mu_1} \otimes \tilde{\mu_2}  ,$
\item[(v)] $(\tilde{\mu_1}\tilde{\cup}\tilde{\mu_2})O_2\tilde{\mu_3}=(\tilde{\mu_1} O_2 \tilde{\mu_3})\tilde{\cup}(\tilde{\mu_2} O_2 \tilde{\mu_3}), $		
\item[(vi)] $(\tilde{\mu_1}\tilde{\cap}\tilde{\mu_2})O_2\tilde{\mu_3}=(\tilde{\mu_1} O_2 \tilde{\mu_3})\tilde{\cap}(\tilde{\mu_2} O_2 \tilde{\mu_3}). $		
\end{enumerate}
\end{prop}
\begin{proof}
(i) $ (\tilde{\mu_1}\oplus\tilde{\mu_2})\tilde{\cap}(\tilde{\mu_1} O_2 \tilde{\mu_2})\\
= (\bigcup_{\gamma_1\tilde{\in}\mu_1,\gamma_2\tilde{\in}\mu_2 }[\tilde{\gamma_1}^L+\tilde{\gamma_2}^L-\tilde{\gamma_1}^L.\tilde{\gamma_2}^L, \tilde{\gamma_1}^U+\tilde{\gamma_2}^U-\tilde{\gamma_1}^U.\tilde{\gamma_2}^U])\tilde{\cap}(\bigcup_{\gamma_1\tilde{\in}\mu_1,\gamma_2\tilde{\in}\mu_2 }[\frac{\mid{\gamma_1 }^L-{\gamma_2 }^L\mid}{1+2\mid{\gamma_1 }^L-{\gamma_2 }^L\mid},\frac{\mid{\gamma_1 }^U-{\gamma_2 }^U\mid}{1+2\mid{\gamma_1 }^U-{\gamma_2 }^U\mid} ])\\
= \bigcup_{\gamma_1\tilde{\in}\mu_1,\gamma_2\tilde{\in}\mu_2 }[\textrm{min}\{\tilde{\gamma_1}^L+\tilde{\gamma_2}^L-\tilde{\gamma_1}^L.\tilde{\gamma_2}^L,\frac{\mid{\gamma_1 }^L-{\gamma_2 }^L\mid}{1+2\mid{\gamma_1 }^L-{\gamma_2 }^L\mid} \},\textrm{min}\{\tilde{\gamma_1}^U+\tilde{\gamma_2}^U-\tilde{\gamma_1}^U.\tilde{\gamma_2}^U,\frac{\mid{\gamma_1 }^U-{\gamma_2 }^U\mid}{1+2\mid{\gamma_1 }^U-{\gamma_2 }^U\mid} \} ]\\
=\bigcup_{\gamma_1\tilde{\in}\mu_1,\gamma_2\tilde{\in}\mu_2 }[\frac{\mid{\gamma_1 }^L-{\gamma_2 }^L\mid}{1+2\mid{\gamma_1 }^L-{\gamma_2 }^L\mid},\frac{\mid{\gamma_1 }^U-{\gamma_2 }^U\mid}{1+2\mid{\gamma_1 }^U-{\gamma_2 }^U\mid}]\\
=\tilde{\mu_1} O_2 \tilde{\mu_2}.$\\
	 	
(ii) $ (\tilde{\mu_1}\oplus\tilde{\mu_2})\tilde{\cup}(\tilde{\mu_1} O_2 \tilde{\mu_2})\\
=(\bigcup_{\gamma_1\tilde{\in}\mu_1,\gamma_2\tilde{\in}\mu_2 }[\tilde{\gamma_1}^L+\tilde{\gamma_2}^L-\tilde{\gamma_1}^L.\tilde{\gamma_2}^L, \tilde{\gamma_1}^U+\tilde{\gamma_2}^U-\tilde{\gamma_1}^U.\tilde{\gamma_2}^U])\tilde{\cup}(\bigcup_{\gamma_1\tilde{\in}\mu_1,\gamma_2\tilde{\in}\mu_2 }[\frac{\mid{\gamma_1 }^L-{\gamma_2 }^L\mid}{1+2\mid{\gamma_1 }^L-{\gamma_2 }^L\mid},\frac{\mid{\gamma_1 }^U-{\gamma_2 }^U\mid}{1+2\mid{\gamma_1 }^U-{\gamma_2 }^U\mid} ])\\
= \bigcup_{\gamma_1\tilde{\in}\mu_1,\gamma_2\tilde{\in}\mu_2 }[\textrm{max}\{\tilde{\gamma_1}^L+\tilde{\gamma_2}^L-\tilde{\gamma_1}^L.\tilde{\gamma_2}^L,\frac{\mid{\gamma_1 }^L-{\gamma_2 }^L\mid}{1+2\mid{\gamma_1 }^L-{\gamma_2 }^L\mid} \},\textrm{max}\{\tilde{\gamma_1}^U+\tilde{\gamma_2}^U-\tilde{\gamma_1}^U.\tilde{\gamma_2}^U,\frac{\mid{\gamma_1 }^U-{\gamma_2 }^U\mid}{1+2\mid{\gamma_1 }^U-{\gamma_2 }^U\mid} \} ]\\
=\bigcup_{\gamma_1\tilde{\in}\mu_1,\gamma_2\tilde{\in}\mu_2 }[\tilde{\gamma_1}^L+\tilde{\gamma_2}^L-\tilde{\gamma_1}^L.\tilde{\gamma_2}^L, \tilde{\gamma_1}^U+\tilde{\gamma_2}^U-\tilde{\gamma_1}^U.\tilde{\gamma_2}^U]\\
=\tilde{\mu_1} \oplus \tilde{\mu_2} .$\\
	 
(iii) $ (\tilde{\mu_1}\otimes\tilde{\mu_2})\tilde{\cap}(\tilde{\mu_1} O_2 \tilde{\mu_2})\\
= (\bigcup_{\gamma_1\tilde{\in}\mu_1,\gamma_2\tilde{\in}\mu_2 }[\tilde{\gamma_1}^L.\tilde{\gamma_2}^L, \tilde{\gamma_1}^U.\tilde{\gamma_2}^U])\tilde{\cap}(\bigcup_{\gamma_1\tilde{\in}\mu_1,\gamma_2\tilde{\in}\mu_2 }[\frac{\mid{\gamma_1 }^L-{\gamma_2 }^L\mid}{1+2\mid{\gamma_1 }^L-{\gamma_2 }^L\mid},\frac{\mid{\gamma_1 }^U-{\gamma_2 }^U\mid}{1+2\mid{\gamma_1 }^U-{\gamma_2 }^U\mid} ])\\
= \bigcup_{\gamma_1\tilde{\in}\mu_1,\gamma_2\tilde{\in}\mu_2 }[\textrm{min}\{\tilde{\gamma_1}^L.\tilde{\gamma_2}^L,\frac{\mid{\gamma_1 }^L-{\gamma_2 }^L\mid}{1+2\mid{\gamma_1 }^L-{\gamma_2 }^L\mid} \},\textrm{min}\{\tilde{\gamma_1}^U.\tilde{\gamma_2}^U,\frac{\mid{\gamma_1 }^U-{\gamma_2 }^U\mid}{1+2\mid{\gamma_1 }^U-{\gamma_2 }^U\mid} \} ]\\
=\bigcup_{\gamma_1\tilde{\in}\mu_1,\gamma_2\tilde{\in}\mu_2 }[\frac{\mid{\gamma_1 }^L-{\gamma_2 }^L\mid}{1+2\mid{\gamma_1 }^L-{\gamma_2 }^L\mid},\frac{\mid{\gamma_1 }^U-{\gamma_2 }^U\mid}{1+2\mid{\gamma_1 }^U-{\gamma_2 }^U\mid}]\\
=\tilde{\mu_1} O_2 \tilde{\mu_2}.$\\

(iv) $ (\tilde{\mu_1}\otimes\tilde{\mu_2})\tilde{\cup}(\tilde{\mu_1} O_2 \tilde{\mu_2})\\
= (\bigcup_{\gamma_1\tilde{\in}\mu_1,\gamma_2\tilde{\in}\mu_2 }[\tilde{\gamma_1}^L.\tilde{\gamma_2}^L, \tilde{\gamma_1}^U.\tilde{\gamma_2}^U])\tilde{\cup}(\bigcup_{\gamma_1\tilde{\in}\mu_1,\gamma_2\tilde{\in}\mu_2 }[\frac{\mid{\gamma_1 }^L-{\gamma_2 }^L\mid}{1+2\mid{\gamma_1 }^L-{\gamma_2 }^L\mid},\frac{\mid{\gamma_1 }^U-{\gamma_2 }^U\mid}{1+2\mid{\gamma_1 }^U-{\gamma_2 }^U\mid} ])\\
= \bigcup_{\gamma_1\tilde{\in}\mu_1,\gamma_2\tilde{\in}\mu_2 }[\textrm{max}\{\tilde{\gamma_1}^L.\tilde{\gamma_2}^L,\frac{\mid{\gamma_1 }^L-{\gamma_2 }^L\mid}{1+2\mid{\gamma_1 }^L-{\gamma_2 }^L\mid} \},\textrm{max}\{\tilde{\gamma_1}^U.\tilde{\gamma_2}^U,\frac{\mid{\gamma_1 }^U-{\gamma_2 }^U\mid}{1+2\mid{\gamma_1 }^U-{\gamma_2 }^U\mid} \} ]\\
=(\bigcup_{\gamma_1\tilde{\in}\mu_1,\gamma_2\tilde{\in}\mu_2 }[\tilde{\gamma_1}^L.\tilde{\gamma_2}^L, \tilde{\gamma_1}^U.\tilde{\gamma_2}^U])\\
=\tilde{\mu_1} \otimes \tilde{\mu_2}.$\\

(v) $(\tilde{\mu_1}\tilde{\cup}\tilde{\mu_2})O_2\tilde{\mu_3}\\	
= \bigcup_{\gamma_1\tilde{\in}\mu_1,\gamma_2\tilde{\in}\mu_2 }[\textrm{max}\{\tilde{\gamma_1}^L,\tilde{\gamma_2}^L\},\textrm{max}\{\tilde{\gamma_1}^U,\tilde{\gamma_2}^U \} ]O_2 \bigcup_{\gamma_3\tilde{\in}\mu_3}[\tilde{\gamma_3}^L, {\gamma_3}^U]\\
=\bigcup_{\gamma_1\tilde{\in}\mu_1,\gamma_2\tilde{\in}\mu_2, \gamma_3\tilde{\in}\mu_3 }[\frac{\mid\textrm{max}\{{\gamma_1 }^L.{\gamma_2 }^L\}-{\gamma_3 }^L\mid}{1+2\mid\textrm{max}\{{\gamma_1 }^L.{\gamma_2 }^L\}-{\gamma_3 }^L\mid},\frac{\mid\textrm{max}\{{\gamma_1 }^U.{\gamma_2 }^U\}-{\gamma_3 }^U\mid}{1+2\mid\textrm{max}\{{\gamma_1 }^U.{\gamma_2 }^U\}-{\gamma_3 }^U\mid}]\\
=\bigcup_{\gamma_1\tilde{\in}\mu_1,\gamma_2\tilde{\in}\mu_2 }[\textrm{max}\{\frac{\mid{\gamma_1 }^L-{\gamma_2 }^L\mid}{1+2\mid{\gamma_1 }^L-{\gamma_2 }^L\mid},\frac{\mid{\gamma_2
	}^L-{\gamma_3 }^L\mid}{1+2\mid{\gamma_2 }^L-{\gamma_3 }^L\mid}\}, \textrm{max}\{\frac{\mid{\gamma_1 }^U-{\gamma_2 }^U\mid}{1+2\mid{\gamma_1 }^U-{\gamma_2 }^U\mid},\frac{\mid{\gamma_2
}^U-{\gamma_3 }^U\mid}{1+2\mid{\gamma_2 }^U-{\gamma_3 }^U\mid}\} ]\\
=(\tilde{\mu_1} O_2 \tilde{\mu_3})\tilde{\cup}(\tilde{\mu_2} O_2 \tilde{\mu_3}). $\\

(vi) $ (\tilde{\mu_1}\tilde{\cap}\tilde{\mu_2})O_2\tilde{\mu_3}\\
= \bigcup_{\gamma_1\tilde{\in}\mu_1,\gamma_2\tilde{\in}\mu_2 }[\textrm{min}\{\tilde{\gamma_1}^L,\tilde{\gamma_2}^L\},\textrm{min}\{\tilde{\gamma_1}^U,\tilde{\gamma_2}^U \} ]O_2 \bigcup_{\gamma_3\tilde{\in}\mu_3}[\tilde{\gamma_3}^L, {\gamma_3}^U]\\
=\bigcup_{\gamma_1\tilde{\in}\mu_1,\gamma_2\tilde{\in}\mu_2, \gamma_3\tilde{\in}\mu_3 }[\frac{\mid\textrm{min}\{{\gamma_1 }^L.{\gamma_2 }^L\}-{\gamma_3 }^L\mid}{1+2\mid\textrm{min}\{{\gamma_1 }^L.{\gamma_2 }^L\}-{\gamma_3 }^L\mid},\frac{\mid\textrm{min}\{{\gamma_1 }^U.{\gamma_2 }^U\}-{\gamma_3 }^U\mid}{1+2\mid\textrm{min}\{{\gamma_1 }^U.{\gamma_2 }^U\}-{\gamma_3 }^U\mid}]\\
=\bigcup_{\gamma_1\tilde{\in}\mu_1,\gamma_2\tilde{\in}\mu_2 }[\textrm{min}\{\frac{\mid{\gamma_1 }^L-{\gamma_2 }^L\mid}{1+2\mid{\gamma_1 }^L-{\gamma_2 }^L\mid},\frac{\mid{\gamma_2
	}^L-{\gamma_3 }^L\mid}{1+2\mid{\gamma_2 }^L-{\gamma_3 }^L\mid}\}, \textrm{min}\{\frac{\mid{\gamma_1 }^U-{\gamma_2 }^U\mid}{1+2\mid{\gamma_1 }^U-{\gamma_2 }^U\mid},\frac{\mid{\gamma_2
}^U-{\gamma_3 }^U\mid}{1+2\mid{\gamma_2 }^U-{\gamma_3 }^U\mid}\} ]\\
=(\tilde{\mu_1} O_2 \tilde{\mu_3})\tilde{\cap}(\tilde{\mu_2} O_2 \tilde{\mu_3}).$
\end{proof}			
\begin{prop} If $ \tilde{\mu_1},  \tilde{\mu_2} $ and $ \tilde{\mu_3} $ be two interval-valued hesitant fuzzy soft elements. Then the following identites are true: 
	\begin{enumerate}
\item[(i)] $(\tilde{\mu_1}\oplus\tilde{\mu_2})\tilde{\cap}(\tilde{\mu_1} O_3 \tilde{\mu_2})=\tilde{\mu_1} O_3 \tilde{\mu_2}  ,$
\item[(ii)] $(\tilde{\mu_1}\oplus\tilde{\mu_2})\tilde{\cup}(\tilde{\mu_1} O_3 \tilde{\mu_2})=\tilde{\mu_1} \oplus \tilde{\mu_2}  ,$
\item[(iii)] $(\tilde{\mu_1}\otimes\tilde{\mu_2})\tilde{\cap}(\tilde{\mu_1} O_3 \tilde{\mu_2})=\tilde{\mu_1} O_3 \tilde{\mu_2}  ,$
\item[(iv)] $(\tilde{\mu_1}\otimes\tilde{\mu_2})\tilde{\cup}(\tilde{\mu_1} O_3 \tilde{\mu_2})=\tilde{\mu_1} \otimes \tilde{\mu_2}  ,$
\item[(v)] $(\tilde{\mu_1}\tilde{\cup}\tilde{\mu_2})O_3\tilde{\mu_3}=(\tilde{\mu_1} O_3 \tilde{\mu_3})\tilde{\cup}(\tilde{\mu_2} O_3 \tilde{\mu_3}), $			
\item[(vi)] $(\tilde{\mu_1}\tilde{\cap}\tilde{\mu_2})O_3\tilde{\mu_3}=(\tilde{\mu_1} O_3 \tilde{\mu_3})\tilde{\cap}(\tilde{\mu_2} O_3 \tilde{\mu_3}). $	
\end{enumerate}
\end{prop}
\begin{proof}

(i) $(\tilde{\mu_1}\oplus\tilde{\mu_2})\tilde{\cap}(\tilde{\mu_1} O_3 \tilde{\mu_2})\\
= (\bigcup_{\gamma_1\tilde{\in}\mu_1,\gamma_2\tilde{\in}\mu_2 }[\tilde{\gamma_1}^L+\tilde{\gamma_2}^L-\tilde{\gamma_1}^L.\tilde{\gamma_2}^L, \tilde{\gamma_1}^U+\tilde{\gamma_2}^U-\tilde{\gamma_1}^U.\tilde{\gamma_2}^U])\tilde{\cap}(\bigcup_{\gamma_1\tilde{\in}\mu_1,\gamma_2\tilde{\in}\mu_2 }[\frac{\mid{\gamma_1 }^L-{\gamma_2 }^L\mid}{2},\frac{\mid{\gamma_1 }^U-{\gamma_2 }^U\mid}{2} ])\\
= \bigcup_{\gamma_1\tilde{\in}\mu_1,\gamma_2\tilde{\in}\mu_2 }[\textrm{min}\{\tilde{\gamma_1}^L+\tilde{\gamma_2}^L-\tilde{\gamma_1}^L.\tilde{\gamma_2}^L,\frac{\mid{\gamma_1 }^L-{\gamma_2 }^L\mid}{2} \},\textrm{min}\{\tilde{\gamma_1}^U+\tilde{\gamma_2}^U-\tilde{\gamma_1}^U.\tilde{\gamma_2}^U,\frac{\mid{\gamma_1 }^U-{\gamma_2 }^U\mid}{2} \} ]\\
=\bigcup_{\gamma_1\tilde{\in}\mu_1,\gamma_2\tilde{\in}\mu_2 }[\frac{\mid{\gamma_1 }^L-{\gamma_2 }^L\mid}{2},\frac{\mid{\gamma_1 }^U-{\gamma_2 }^U\mid}{2}]\\
=\tilde{\mu_1} O_3 \tilde{\mu_2}. $\\

(ii) $(\tilde{\mu_1}\oplus\tilde{\mu_2})\tilde{\cup}(\tilde{\mu_1} O_3 \tilde{\mu_2})\\
=(\bigcup_{\gamma_1\tilde{\in}\mu_1,\gamma_2\tilde{\in}\mu_2 }[\tilde{\gamma_1}^L+\tilde{\gamma_2}^L-\tilde{\gamma_1}^L.\tilde{\gamma_2}^L, \tilde{\gamma_1}^U+\tilde{\gamma_2}^U-\tilde{\gamma_1}^U.\tilde{\gamma_2}^U])\tilde{\cup}(\bigcup_{\gamma_1\tilde{\in}\mu_1,\gamma_2\tilde{\in}\mu_2 }[\frac{\mid{\gamma_1 }^L-{\gamma_2 }^L\mid}{2},\frac{\mid{\gamma_1 }^U-{\gamma_2 }^U\mid}{2} ])\\
= \bigcup_{\gamma_1\tilde{\in}\mu_1,\gamma_2\tilde{\in}\mu_2 }[\textrm{max}\{\tilde{\gamma_1}^L+\tilde{\gamma_2}^L-\tilde{\gamma_1}^L.\tilde{\gamma_2}^L,\frac{\mid{\gamma_1 }^L-{\gamma_2 }^L\mid}{2} \},\textrm{max}\{\tilde{\gamma_1}^U+\tilde{\gamma_2}^U-\tilde{\gamma_1}^U.\tilde{\gamma_2}^U,\frac{\mid{\gamma_1 }^U-{\gamma_2 }^U\mid}{2} \} ]\\
=\bigcup_{\gamma_1\tilde{\in}\mu_1,\gamma_2\tilde{\in}\mu_2 }[\tilde{\gamma_1}^L+\tilde{\gamma_2}^L-\tilde{\gamma_1}^L.\tilde{\gamma_2}^L, \tilde{\gamma_1}^U+\tilde{\gamma_2}^U-\tilde{\gamma_1}^U.\tilde{\gamma_2}^U]\\
=\tilde{\mu_1} \oplus \tilde{\mu_2} .$	\\
  
(iii) $(\tilde{\mu_1}\otimes\tilde{\mu_2})\tilde{\cap}(\tilde{\mu_1} O_3 \tilde{\mu_2})\\
= (\bigcup_{\gamma_1\tilde{\in}\mu_1,\gamma_2\tilde{\in}\mu_2 }[\tilde{\gamma_1}^L.\tilde{\gamma_2}^L, \tilde{\gamma_1}^U.\tilde{\gamma_2}^U])\tilde{\cap}(\bigcup_{\gamma_1\tilde{\in}\mu_1,\gamma_2\tilde{\in}\mu_2 }[\frac{\mid{\gamma_1 }^L-{\gamma_2 }^L\mid}{2},\frac{\mid{\gamma_1 }^U-{\gamma_2 }^U\mid}{2} ])\\
= \bigcup_{\gamma_1\tilde{\in}\mu_1,\gamma_2\tilde{\in}\mu_2 }[\textrm{min}\{\tilde{\gamma_1}^L.\tilde{\gamma_2}^L,\frac{\mid{\gamma_1 }^L-{\gamma_2 }^L\mid}{2} \},\textrm{min}\{\tilde{\gamma_1}^U.\tilde{\gamma_2}^U,\frac{\mid{\gamma_1 }^U-{\gamma_2 }^U\mid}{2} \} ]\\
=\bigcup_{\gamma_1\tilde{\in}\mu_1,\gamma_2\tilde{\in}\mu_2 }[\frac{\mid{\gamma_1 }^L-{\gamma_2 }^L\mid}{2},\frac{\mid{\gamma_1 }^U-{\gamma_2 }^U\mid}{2}]\\
=\tilde{\mu_1} O_3 \tilde{\mu_2} . $\\

(iv) $ (\tilde{\mu_1}\otimes\tilde{\mu_2})\tilde{\cup}(\tilde{\mu_1} O_3 \tilde{\mu_2})\\
=(\bigcup_{\gamma_1\tilde{\in}\mu_1,\gamma_2\tilde{\in}\mu_2 }[\tilde{\gamma_1}^L.\tilde{\gamma_2}^L, \tilde{\gamma_1}^U.\tilde{\gamma_2}^U])\tilde{\cup}(\bigcup_{\gamma_1\tilde{\in}\mu_1,\gamma_2\tilde{\in}\mu_2 }[\frac{\mid{\gamma_1 }^L-{\gamma_2 }^L\mid}{2},\frac{\mid{\gamma_1 }^U-{\gamma_2 }^U\mid}{2} ])\\
= \bigcup_{\gamma_1\tilde{\in}\mu_1,\gamma_2\tilde{\in}\mu_2 }[\textrm{max}\{\tilde{\gamma_1}^L.\tilde{\gamma_2}^L,\frac{\mid{\gamma_1 }^L-{\gamma_2 }^L\mid}{2} \},\textrm{max}\{\tilde{\gamma_1}^U.\tilde{\gamma_2}^U,\frac{\mid{\gamma_1 }^U-{\gamma_2 }^U\mid}{2} \} ]\\
=(\bigcup_{\gamma_1\tilde{\in}\mu_1,\gamma_2\tilde{\in}\mu_2 }[\tilde{\gamma_1}^L.\tilde{\gamma_2}^L, \tilde{\gamma_1}^U.\tilde{\gamma_2}^U])\\
=\tilde{\mu_1} \otimes \tilde{\mu_2}.  $\\

(v) $(\tilde{\mu_1}\tilde{\cup}\tilde{\mu_2})O_3\tilde{\mu_3}\\
= \bigcup_{\gamma_1\tilde{\in}\mu_1,\gamma_2\tilde{\in}\mu_2 }[\textrm{max}\{\tilde{\gamma_1}^L,\tilde{\gamma_2}^L\},\textrm{max}\{\tilde{\gamma_1}^U,\tilde{\gamma_2}^U \} ]O_3 \bigcup_{\gamma_3\tilde{\in}\mu_3}[\tilde{\gamma_3}^L, {\gamma_3}^U]\\
=\bigcup_{\gamma_1\tilde{\in}\mu_1,\gamma_2\tilde{\in}\mu_2, \gamma_3\tilde{\in}\mu_3 }[\frac{\mid\textrm{max}\{{\gamma_1 }^L.{\gamma_2 }^L\}-{\gamma_3 }^L\mid}{2},\frac{\mid\textrm{max}\{{\gamma_1 }^U.{\gamma_2 }^U\}-{\gamma_3 }^U\mid}{2}]\\
=\bigcup_{\gamma_1\tilde{\in}\mu_1,\gamma_2\tilde{\in}\mu_2,\gamma_3\tilde{\in}\mu_3 }[\textrm{max}\{\frac{\mid{\gamma_1 }^L-{\gamma_2 }^L\mid}{2},\frac{\mid{\gamma_2
	}^L-{\gamma_3 }^L\mid}{2}\}, \textrm{max}\{\frac{\mid{\gamma_1 }^U-{\gamma_2 }^U\mid}{2},\frac{\mid{\gamma_2
}^U-{\gamma_3 }^U\mid}{2}\} ]\\
=(\tilde{\mu_1} O_3 \tilde{\mu_3})\tilde{\cup}(\tilde{\mu_2} O_3 \tilde{\mu_3}). $\\

(vi) $ (\tilde{\mu_1}\tilde{\cap}\tilde{\mu_2})O_3\tilde{\mu_3}\\
=\bigcup_{\gamma_1\tilde{\in}\mu_1,\gamma_2\tilde{\in}\mu_2 }[\textrm{min}\{\tilde{\gamma_1}^L,\tilde{\gamma_2}^L\},\textrm{min}\{\tilde{\gamma_1}^U,\tilde{\gamma_2}^U \} ]O_3 \bigcup_{\gamma_3\tilde{\in}\mu_3}[\tilde{\gamma_3}^L, {\gamma_3}^U]\\
=\bigcup_{\gamma_1\tilde{\in}\mu_1,\gamma_2\tilde{\in}\mu_2, \gamma_3\tilde{\in}\mu_3 }[\frac{\mid\textrm{min}\{{\gamma_1 }^L.{\gamma_2 }^L\}-{\gamma_3 }^L\mid}{2},\frac{\mid\textrm{min}\{{\gamma_1 }^U.{\gamma_2 }^U\}-{\gamma_3 }^U\mid}{2}]\\
=\bigcup_{\gamma_1\tilde{\in}\mu_1,\gamma_2\tilde{\in}\mu_2,\gamma_3\tilde{\in}\mu_3 }[\textrm{min}\{\frac{\mid{\gamma_1 }^L-{\gamma_2 }^L\mid}{2},\frac{\mid{\gamma_2
	}^L-{\gamma_3 }^L\mid}{2}\}, \textrm{min}\{\frac{\mid{\gamma_1 }^U-{\gamma_2 }^U\mid}{2},\frac{\mid{\gamma_2
}^U-{\gamma_3 }^U\mid}{2}\} ]\\
=(\tilde{\mu_1} O_3 \tilde{\mu_3})\tilde{\cap}(\tilde{\mu_2} O_3 \tilde{\mu_3}).  $
\end{proof}			
\begin{prop} If $ \tilde{\mu_1},  \tilde{\mu_2} $ and $ \tilde{\mu_3} $ be two interval-valued hesitant fuzzy soft elements. Then the following identites are true: 
	\begin{enumerate}
\item[(i)] $(\tilde{\mu_1}\oplus\tilde{\mu_2})\tilde{\cap}(\tilde{\mu_1} O_4 \tilde{\mu_2})=\tilde{\mu_1} O_4 \tilde{\mu_2}  ,$
\item[(ii)] $(\tilde{\mu_1}\oplus\tilde{\mu_2})\tilde{\cup}(\tilde{\mu_1} O_4 \tilde{\mu_2})=\tilde{\mu_1} \oplus \tilde{\mu_2}  ,$
\item[(iii)] $(\tilde{\mu_1}\otimes\tilde{\mu_2})\tilde{\cap}(\tilde{\mu_1} O_4 \tilde{\mu_2})=\tilde{\mu_1} O_4 \tilde{\mu_2}  ,$	
\item[(iv)] $(\tilde{\mu_1}\otimes\tilde{\mu_2})\tilde{\cup}(\tilde{\mu_1} O_4 \tilde{\mu_2})=\tilde{\mu_1} \otimes \tilde{\mu_2}  ,$
\item[(v)] $(\tilde{\mu_1}\tilde{\cup}\tilde{\mu_2})O_4\tilde{\mu_3}=(\tilde{\mu_1} O_4 \tilde{\mu_3})\tilde{\cup}(\tilde{\mu_2} O_4 \tilde{\mu_3}), $			
\item[(vi)] $(\tilde{\mu_1}\tilde{\cap}\tilde{\mu_2})O_4\tilde{\mu_3}=(\tilde{\mu_1} O_4 \tilde{\mu_3})\tilde{\cap}(\tilde{\mu_2} O_4 \tilde{\mu_3}). $	
\end{enumerate}
\end{prop}
\begin{proof}		
(i) $(\tilde{\mu_1}\oplus\tilde{\mu_2})\tilde{\cap}(\tilde{\mu_1} O_4 \tilde{\mu_2})\\
=(\bigcup_{\gamma_1\tilde{\in}\mu_1,\gamma_2\tilde{\in}\mu_2 }[\tilde{\gamma_1}^L+\tilde{\gamma_2}^L-\tilde{\gamma_1}^L.\tilde{\gamma_2}^L, \tilde{\gamma_1}^U+\tilde{\gamma_2}^U-\tilde{\gamma_1}^U.\tilde{\gamma_2}^U])\tilde{\cap}(\bigcup_{\gamma_1\tilde{\in}h_1,\gamma_2\tilde{\in}h_2 }[\frac{\mid{\gamma_1 }^L.{\gamma_2 }^L\mid}{2},\frac{\mid{\gamma_1 }^U.{\gamma_2 }^U\mid}{2} ])\\
= \bigcup_{\gamma_1\tilde{\in}\mu_1,\gamma_2\tilde{\in}\mu_2 }[\textrm{min}\{\tilde{\gamma_1}^L+\tilde{\gamma_2}^L-\tilde{\gamma_1}^L.\tilde{\gamma_2}^L,\frac{\mid{\gamma_1 }^L.{\gamma_2 }^L\mid}{2} \},\textrm{min}\{\tilde{\gamma_1}^U+\tilde{\gamma_2}^U-\tilde{\gamma_1}^U.\tilde{\gamma_2}^U,\frac{\mid{\gamma_1 }^U.{\gamma_2 }^U\mid}{2} \} ]\\
=\bigcup_{\gamma_1\tilde{\in}\mu_1,\gamma_2\tilde{\in}\mu_2 }[\frac{\mid{\gamma_1 }^L.{\gamma_2 }^L\mid}{2},\frac{\mid{\gamma_1 }^U.{\gamma_2 }^U\mid}{2}]\\
=\tilde{\mu_1} O_4 \tilde{\mu_2}.   $\\

(ii) $(\tilde{\mu_1}\oplus\tilde{\mu_2})\tilde{\cup}(\tilde{\mu_1} O_4 \tilde{\mu_2})\\
=(\bigcup_{\gamma_1\tilde{\in}\mu_1,\gamma_2\tilde{\in}\mu_2 }[\tilde{\gamma_1}^L+\tilde{\gamma_2}^L-\tilde{\gamma_1}^L.\tilde{\gamma_2}^L, \tilde{\gamma_1}^U+\tilde{\gamma_2}^U-\tilde{\gamma_1}^U.\tilde{\gamma_2}^U])\tilde{\cup}(\bigcup_{\gamma_1\tilde{\in}\mu_1,\gamma_2\tilde{\in}\mu_2 }[\frac{\mid{\gamma_1 }^L.{\gamma_2 }^L\mid}{2},\frac{\mid{\gamma_1 }^U.{\gamma_2 }^U\mid}{2} ])\\
= \bigcup_{\gamma_1\tilde{\in}\mu_1,\gamma_2\tilde{\in}\mu_2 }[\textrm{max}\{\tilde{\gamma_1}^L+\tilde{\gamma_2}^L-\tilde{\gamma_1}^L.\tilde{\gamma_2}^L,\frac{\mid{\gamma_1 }^L.{\gamma_2 }^L\mid}{2} \},\textrm{max}\{\tilde{\gamma_1}^U+\tilde{\gamma_2}^U-\tilde{\gamma_1}^U.\tilde{\gamma_2}^U,\frac{\mid{\gamma_1 }^U.{\gamma_2 }^U\mid}{2} \} ]\\
=\bigcup_{\gamma_1\tilde{\in}\mu_1,\gamma_2\tilde{\in}\mu_2 }[\tilde{\gamma_1}^L+\tilde{\gamma_2}^L-\tilde{\gamma_1}^L.\tilde{\gamma_2}^L, \tilde{\gamma_1}^U+\tilde{\gamma_2}^U-\tilde{\gamma_1}^U.\tilde{\gamma_2}^U]\\
=\tilde{\mu_1} \oplus \tilde{\mu_2} .  $\\

(iii) $ (\tilde{\mu_1}\otimes\tilde{\mu_2})\tilde{\cap}(\tilde{\mu_1} O_4 \tilde{\mu_2})\\
=(\bigcup_{\gamma_1\tilde{\in}\mu_1,\gamma_2\tilde{\in}\mu_2 }[\tilde{\gamma_1}^L.\tilde{\gamma_2}^L, \tilde{\gamma_1}^U.\tilde{\gamma_2}^U])\tilde{\cap}(\bigcup_{\gamma_1\tilde{\in}\mu_1,\gamma_2\tilde{\in}\mu_2 }[\frac{\mid{\gamma_1 }^L.{\gamma_2 }^L\mid}{2},\frac{\mid{\gamma_1 }^U.{\gamma_2 }^U\mid}{2} ])\\
= \bigcup_{\gamma_1\tilde{\in}\mu_1,\gamma_2\tilde{\in}\mu_2 }[\textrm{min}\{\tilde{\gamma_1}^L.\tilde{\gamma_2}^L,\frac{\mid{\gamma_1 }^L.{\gamma_2 }^L\mid}{2} \},\textrm{min}\{\tilde{\gamma_1}^U.\tilde{\gamma_2}^U,\frac{\mid{\gamma_1 }^U.{\gamma_2 }^U\mid}{2} \} ]\\
=\bigcup_{\gamma_1\tilde{\in}\mu_1,\gamma_2\tilde{\in}\mu_2 }[\frac{\mid{\gamma_1 }^L.{\gamma_2 }^L\mid}{2},\frac{\mid{\gamma_1 }^U.{\gamma_2 }^U\mid}{2}]\\
=\tilde{\mu_1} O_4 \tilde{\mu_2} .  $\\

(iv) $(\tilde{\mu_1}\otimes\tilde{\mu_2})\tilde{\cup}(\tilde{\mu_1} O_4 \tilde{\mu_2})\\
=(\bigcup_{\gamma_1\tilde{\in}\mu_1,\gamma_2\tilde{\in}\mu_2 }[\tilde{\gamma_1}^L.\tilde{\gamma_2}^L, \tilde{\gamma_1}^U.\tilde{\gamma_2}^U])\tilde{\cup}(\bigcup_{\gamma_1\tilde{\in}\mu_1,\gamma_2\tilde{\in}\mu_2 }[\frac{\mid{\gamma_1 }^L.{\gamma_2 }^L\mid}{2},\frac{\mid{\gamma_1 }^U.{\gamma_2 }^U\mid}{2} ])\\
= \bigcup_{\gamma_1\tilde{\in}\mu_1,\gamma_2\tilde{\in}\mu_2 }[\textrm{max}\{\tilde{\gamma_1}^L.\tilde{\gamma_2}^L,\frac{\mid{\gamma_1 }^L.{\gamma_2 }^L\mid}{2} \},\textrm{max}\{\tilde{\gamma_1}^U.\tilde{\gamma_2}^U,\frac{\mid{\gamma_1 }^U.{\gamma_2 }^U\mid}{2} \} ]\\
=(\bigcup_{\gamma_1\tilde{\in}\mu_1,\gamma_2\tilde{\in}\mu_2 }[\tilde{\gamma_1}^L.\tilde{\gamma_2}^L, \tilde{\gamma_1}^U.\tilde{\gamma_2}^U])\\
=\tilde{\mu_1} \otimes \tilde{\mu_2}.   $	\\

(v) $(\tilde{\mu_1}\tilde{\cup}\tilde{\mu_2})O_4\tilde{\mu_3}\\
= \bigcup_{\gamma_1\tilde{\in}\mu_1,\gamma_2\tilde{\in}\mu_2 }[\textrm{max}\{\tilde{\gamma_1}^L,\tilde{\gamma_2}^L\},\textrm{max}\{\tilde{\gamma_1}^U,\tilde{\gamma_2}^U \} ]O_4 \bigcup_{\gamma_3\tilde{\in}h_3}[\tilde{\gamma_3}^L, {\gamma_3}^U]\\
=\bigcup_{\gamma_1\tilde{\in}\mu_1,\gamma_2\tilde{\in}\mu_2, \gamma_3\tilde{\in}\mu_3 }[\frac{\mid\textrm{max}\{{\gamma_1 }^L.{\gamma_2 }^L\}\ast{\gamma_3 }^L\mid}{2},\frac{\mid\textrm{max}\{{\gamma_1 }^U.{\gamma_2 }^U\}\ast{\gamma_3 }^U\mid}{2}]\\
=\bigcup_{\gamma_1\tilde{\in}\mu_1,\gamma_2\tilde{\in}\mu_2,\gamma_3\tilde{\in}\mu_3 }[\textrm{max}\{\frac{\mid{\gamma_1 }^L\ast{\gamma_2 }^L\mid}{2},\frac{\mid{\gamma_2
	}^L\ast{\gamma_3 }^L\mid}{2}\}, \textrm{max}\{\frac{\mid{\gamma_1 }^U\ast{\gamma_2 }^U\mid}{2},\frac{\mid{\gamma_2
}^U\ast{\gamma_3 }^U\mid}{2}\} ]\\
=(\tilde{\mu_1} O_4 \tilde{\mu_3})\tilde{\cup}(\tilde{\mu_2} O_4 \tilde{\mu_3}). $\\

(vi) $ (\tilde{\mu_1}\tilde{\cap}\tilde{\mu_2})O_4\tilde{\mu_3}\\
=\bigcup_{\gamma_1\tilde{\in}\mu_1,\gamma_2\tilde{\in}\mu_2 }[\textrm{min}\{\tilde{\gamma_1}^L,\tilde{\gamma_2}^L\},\textrm{min}\{\tilde{\gamma_1}^U,\tilde{\gamma_2}^U \} ]O_4 \bigcup_{\gamma_3\tilde{\in}\mu_3}[\tilde{\gamma_3}^L, {\gamma_3}^U]\\
=\bigcup_{\gamma_1\tilde{\in}\mu_1,\gamma_2\tilde{\in}\mu_2, \gamma_3\tilde{\in}\mu_3 }[\frac{\mid\textrm{min}\{{\gamma_1 }^L.{\gamma_2 }^L\}\ast{\gamma_3 }^L\mid}{2},\frac{\mid\textrm{min}\{{\gamma_1 }^U.{\gamma_2 }^U\}\ast{\gamma_3 }^U\mid}{2}]\\
=\bigcup_{\gamma_1\tilde{\in}\mu_1,\gamma_2\tilde{\in}\mu_2,\gamma_3\tilde{\in}\mu_3 }[\textrm{min}\{\frac{\mid{\gamma_1 }^L\ast{\gamma_2 }^L\mid}{2},\frac{\mid{\gamma_2
	}^L\ast{\gamma_3 }^L\mid}{2}\}, \textrm{min}\{\frac{\mid{\gamma_1 }^U\ast{\gamma_2 }^U\mid}{2},\frac{\mid{\gamma_2
}^U\ast{\gamma_3 }^U\mid}{2}\} ]\\
=(\tilde{\mu_1} O_4 \tilde{\mu_3})\tilde{\cap}(\tilde{\mu_2} O_4 \tilde{\mu_3}). $	
\end{proof}	


\end{document}